\documentclass[10pt]{amsart}
\usepackage{amssymb,amsmath,amsopn,textcomp, thmtools}
\usepackage[foot]{amsaddr}
\newtheorem{satz}{Theorem}[section]

\newtheorem{lemma}{Lemma}[section]

\newtheorem{bemerk1}{Remark}[section]

\newtheorem{bei}{Example}[section]

\newtheorem{korollar}{Corollary}[section]
\newtheorem{corollary}{Corollary}[section]
\newtheorem{proposition}{Proposition}[section]
\newcommand{\iR}{\mathbb{R}}
\newcommand{\iN}{\mathbb{N}}
\newcommand{\iZ}{\mathbb{Z}}

\usepackage{bbm}
\newcommand{\R}{\mathbb{R}}
\newcommand{\N}{\mathbb{N}}

\newcommand{\cL}{\mathcal{L}}
\newcommand{\cA}{\mathcal{A}}

\usepackage{extarrows}
\usepackage{dsfont}
\usepackage{tikz,subcaption} 

\usepackage{mathabx} 

\allowdisplaybreaks

\begin{document}

\title{Curvature-dimension inequalities for non-local operators in the discrete setting}
\date{\today}
\author{Adrian Spener}
\email{adrian.spener@uni-ulm.de}
\author{Frederic Weber}
\email{frederic.weber@uni-ulm.de}
\author{Rico Zacher$^*$}
\thanks{$^*$Corresponding author}
\email[Corresponding author:]{rico.zacher@uni-ulm.de}
\address[Adrian Spener, Frederic Weber, Rico Zacher]{Institut f\"ur Angewandte Analysis, Universit\"at Ulm, Helmholtzstra\ss{}e 18, 89081 Ulm, Germany.}


\begin{abstract}
 We study Bakry-\'Emery curvature-dimension inequalities for non-local operators on the one-dimensional lattice and prove that operators with finite second moment have finite dimension. Moreover, we show that a class of operators related to the fractional Laplacian fails to have finite dimension and establish both positive and negative results for operators with sparsely supported kernels. Furthermore, a large class of operators is shown to have no positive curvature. The results correspond to CD inequalities on locally infinite graphs.
\end{abstract}

\maketitle

 

\section{Introduction and main results}
The main purpose of this paper is to study curvature-dimension (CD) inequalities for  
non-local operators $\cL$ on the lattice $\iZ$ of the form
\begin{equation} \label{genlaplacedef}
\cL v(x)=\,\sum_{j\in \iZ} k(j) \big(v(x+j)-v(x)\big),\quad x\in \iZ,
\end{equation}
with a (nontrivial) kernel $k$ which is nonnegative, integrable and symmetric, that is
\begin{itemize}
\item[(K1)] $k:\,\iZ\to [0,\infty)$, $\sum_{j\in \iZ}k(j)<\infty$ and $k(-j)=k(j)$ for all $j\in \iN$.
\end{itemize}
Observe that the value of $k$ at $0$ does not play a role in the definition
of $\cL$. It is convenient to assume that 
\begin{itemize}
\item[(K2)] $k(0)=0$. 
\end{itemize}
An important example, which will be investigated in this paper, is given by
\begin{equation} \label{powerkerneldef}
k(j)=\,\frac{c}{|j|^{1+\beta}},\quad j\in \iZ\setminus \{0\},
\end{equation}
where $c,\beta>0$.

Curvature-dimension inequalities (or conditions) play a central role in the study of functional inequalities associated to Markov
semigroups and operators. Important examples of such functional inequalities are the Poincar\'e or spectral gap inequality,
the logarithmic Sobolev inequality and the Sobolev inequality, which, among others, allow to derive various estimates of 
solutions to related evolution equations, e.g.\ Harnack inequalities or bounds which imply the exponentially fast trend to an equilibrium. The special feature of CD-inequalities
is that they provide a very useful link to the geometric properties (like dimension and curvature) of the underlying structure (\cite{MR3155209}). For this reason they also constitute an important tool in geometric analysis 
(\cite{MR2962229}).

There are several different notions of CD-inequalities. Here we use the original one, which goes back to Bakry and \'Emery and
is formulated in terms of the carr\'e du champ operator $\Gamma$ and the iterated carr\'e du champ operator $\Gamma_2$
associated with the infinitesimal generator $\cL$ of a Markov semigroup, see \cite{MR889476}. Another powerful approach is based on the theory of optimal transport and displacement convexity inequalities (\cite{MR2237206,MR2237207,MR2459454,MR2480619}). For Riemannian manifolds it provides, like the Bakry-\'Emery calculus, an equivalent definition of Ricci curvature lower bounds. 

Let $\cL$ be the generator of a symmetric Markov semigroup with invariant reversible measure $\mu$ on the state space $E$. The bilinear operators $\Gamma$ and $\Gamma_2$ are defined by
\begin{align*}
\Gamma(u,v)& = \frac{1}{2}\left( \cL(uv) - u \cL v - v \cL u\right),\\
\Gamma_2(u,v)& = \frac{1}{2} \left( \cL \Gamma (u,v) - \Gamma (u, \cL v) - \Gamma (\cL u,v)\right)
\end{align*}
on a suitable algebra $\cA$ of real-valued functions $u,v$ defined on the underlying state space $E$. 
Furthermore, one sets
\begin{align*}
 \Gamma(u) & := \Gamma(u,u) = \frac{1}{2} \cL  (u^2) - u \cL u,\\
 \Gamma_2(u)& := \Gamma_2(u,u) = \frac{1}{2} \cL \Gamma (u) - \Gamma (u, \cL u).
\end{align*}
 
Let $\kappa\in \iR$ and $d \in (0, \infty]$. We say that $\cL$ satisfies the \emph{Bakry-\'Emery curvature-dimension inequality} $CD(\kappa, d)$ with dimension $d$ and curvature $\kappa$ (lower bound) at $x\in E$ if 
\begin{equation} \label{CDBed}
 \Gamma_2(u)(x) \geq \frac{1}{d} \big((\cL u)(x)\big)^2 + \kappa \Gamma(u)(x)
\end{equation}
for all functions $u:\,E\to \iR$ in a sufficiently rich class $\cA$ of functions. We further say that $\cL$ satisfies the
$CD(\kappa,d)$-inequality, if \eqref{CDBed} holds $\mu$-almost everywhere for all $u\in \cA$, cf.\ \cite[Sect. 1.16]{MR3155209}.

{As an illustrating example we let $E = (M,g)$ be a Riemannian manifold with canonical Riemannian measure $\mu_g$ and $\mathcal{L} = \Delta_g$ the Laplace-Beltrami operator. Using the Bochner-Lichnerowicz formula one obtains that $CD(\kappa, d)$ is equivalent to $\operatorname{Ric}_g(x) \geq \kappa g(x)$ and $\dim M \leq d$.}

In our case, that is, the operator $\cL$ is given by \eqref{genlaplacedef}, we have a countable Markov chain with state space $E=\iZ$ and
$\cL$ is a Markov generator, which can be also written as
\[
\cL u(x)=\sum_{y\in \iZ}l(x,y) u(y),\quad x\in \iZ,
\] 
where
\[
l(x,y)=|k|_1\big (p(x,y)-\delta(x,y)\big),\quad p(x,y)=\frac{1}{|k|_1}\,k(x-y),\quad x,y\in \iZ,
\]
and $\delta(x,y)=1$ if $x=y$ and $\delta(x,y)=0$ otherwise. The infinite matrix $(p(x,y))_{(x,y)\in \iZ^2}$ represents 
the transition probabilities of the Markov chain; $p(x,y)$ is the probability to jump from $x$ to $y$ in the next time step. 
The counting measure on $\iZ$ plays the role of the invariant reversible measure $\mu$. A possible choice for the algebra $\cA$ is the space of all bounded functions $l_\infty(\iZ)$. A straight-forward computation shows that
\begin{align}
\Gamma(u)(x) & =\,\frac{1}{2}\,\sum_{j\in \iZ} k(j) \big( u(x+j)-u(x)\big)^2 ,\quad x\in \iZ,\label{Gamma1Formel}\\
\Gamma_2(u)(x) & =\,\frac{1}{4}\,\sum_{j,l\in \iZ} k(j)k(l) \big(u(x+j+l)-u(x+j)-u(x+l)+u(x)\big)^2,\quad x\in \iZ,
\label{Gamma2Formel}
\end{align}
in particular $\Gamma_2(u)\ge 0$, which implies that the $CD(0,\infty)$-inequality is always true. 

The main objective of this paper is to analyse if and under which conditions the operator $\cL$ from  
\eqref{genlaplacedef} satisfies $CD(\kappa,d)$, the main focus lying on $CD(0,d)$-conditions with finite 
dimension $d>0$. Throughout this paper, we assume that the kernel $k$ in \eqref{genlaplacedef} is subject to the conditions
(K1) and (K2) from above and that $|k|_1>0$. We are especially interested in kernels with {\em unbounded support}
as it is the case, e.g., for the algebraic (or power type) kernel $k$ given by \eqref{powerkerneldef}. In this situation,
arbitrary long jumps are possible. The graphs underlying the Markov chain are {\em not locally finite}, in contrast
to related known results in the literature. 

{Let us explain the connection to graph theory in more detail. Consider the undirected, weighted graph $\mathcal G$ given by the vertices $V = \iZ$ with edge weight $\omega(x,y) = k(x-y)$, $x,y \in V$. Setting the weight on the vertices to be constant we obtain \eqref{genlaplacedef} for the graph Laplacian on $\mathcal{G}$. Note that $\mathcal G$ is locally infinite if and only if $k$ has unbounded support.
Curvature-dimension inequalites in the discrete setting, in particular for (locally) finite graphs are studied intensively (\cite{MR3727131,MR3592766,MR3776357, 2018arXiv180710181K}). 
Various examples of CD-inequalities in the sense of \eqref{CDBed} for finite graphs are given in \cite{MR3492631, MR3611318}. Some results on locally finite graphs from \cite{MR2644381, MR3164168} can be adapted to the locally infinite case (see Proposition \ref{prop:CD(-k,2)}). Related notions of curvature-dimension conditions on graphs, e.g.\ so-called exponential
curvature-dimension inequalities, are examined in 
\cite{MR2484937}, \cite{2014arXiv1412.3340M}, \cite{MR3316971} and \cite{2017arXiv170104807D}.
} Curvature-dimension inequalities in the context of optimal transport in the discrete setting have been studied in \cite{MR2824578}, \cite{MR2989449}, \cite{MR3224294}, \cite{MR3385639}, \cite{MR3706773}, \cite{MR3782987}.

We now describe the main results of this paper. The first basic observation is that for any kernel with {\em finite support}
the corresponding operator $\cL$ satisfies $CD(0,2N)$, where $2N$ equals the cardinality of the support, see Theorem 
\ref{finite support}. 
We also prove that the constant $2N$ in the CD-inequality is in general the best one can get.

Turning to the case of unbounded support, we are able to show that for the {\em power type kernel} $k(j)=c|j|^{-(1+\beta)}$
with $c,\beta>0$ the value $\beta=2$ is a critical case. For all $\beta\in (0,2)$ the $CD(0,d)$-condition fails to hold for
all finite $d>0$ (see Theorem \ref{betalessthan2}), whereas for any $\beta\in (2,\infty)$ there exists a finite $d>0$ such that $CD(0,d)$ is valid, cf.\ Theorem \ref{posresult}. In the proof of Theorem \ref{betalessthan2}, we consider the family
of {\em unbounded} functions $u_\varepsilon(j)=|j|^{\beta-\varepsilon}$ ($\varepsilon>0$) and prove that 
$\Gamma_2(u_\varepsilon)(0)/(\cL(u_\varepsilon)(0))^2 \to 0$ as $\varepsilon\to 0$. One may ask whether a
$CD(0,d)$-condition with $d\in (0,\infty)$ is still true on a smaller class of admissible functions {such as} $l_\infty(\iZ)$ or
{
}the set of compactly supported functions. This is not the case as we show by means of an appropriate family of functions with compact support, see Theorem \ref{theo:approximation}. The basic idea of the construction of these functions is 
to use the family $(u_\varepsilon)_{\varepsilon>0}$ from before and carefully chosen cut-off functions, a crude cut-off
of $(u_\varepsilon)_{\varepsilon>0}$ does not seem to work.
The proof is rather technical as the estimate
of the $\Gamma_2$-term requires distinguishing of several cases due to the double series.

An important consequence of the negative result for $\beta<2$ is that for all {\em fractional powers} $\cL=-(-\Delta)^{\beta/2}$ of the {\em discrete Laplacian} $\Delta$ on $\iZ$ with $\beta\in (0,2)$ the $CD(0,d)$-inequality fails to hold as well for all finite $d>0$. Here, the
discrete Laplacian (which satisfies $CD(0,2)$, {see Theorem \ref{finite support}}) is given by
\begin{equation}
\Delta u(x)=u(x+1)-2 u(x)+u(x-1),\quad x\in \iZ, \label{eq:DiscreteLaplacian}
 \end{equation}
and the fractional powers of the discrete Laplacian can be defined by means of the semigroup method (see {\cite[(1.3)]{MR3787555}})
as
\[
-(-\Delta)^{\frac{\beta}{2}}u(x)=\frac{1}{\Gamma(-\frac{\beta}{2})}\int_0^\infty \big(u-e^{t\Delta}u\big)(x)\frac{dt}
{t^{1+\frac{\beta}{2}}},\quad x\in \iZ.
\]

In a very recent work {\cite{Gegenbeispiel}}, we are able to adapt the discrete estimates and the construction of the counterexamples from the case $\beta<2$ to the {\em continuous case}. Employing additional techniques to extend the results to the multi-dimensional case
we can show that the {\em fractional Laplacian on} $\iR^N$ with any $N\in \iN$ fails to satisfy the $CD(0,d)$-condition for all finite $d>0$,
even in the case where only compactly supported $C^\infty$-functions are admissible.

The main positive result, Theorem \ref{posresult}, is not only formulated for purely algebraic kernels with $\beta>2$ but
for a much larger class of sufficiently fast decaying kernels. Besides monotonicity on $\iN$ the central assumption is that
the kernel $k$ has a {\em finite second moment}, that is,
\[
\sum_{j\in \iN} k(j)j^2<\infty.
\]
{This class also
includes, for example,} all kernels of the form
\begin{equation} \label{kbeispiel}
k(j)=c \,\frac{e^{- \delta |j|^\alpha}}{|j|^\gamma},\quad j\in \iZ\setminus \{0\},
\end{equation}
where $c,\alpha,\delta>0$ and $\gamma\ge 0$, in particular exponential kernels. We also show that
the statement of Theorem \ref{posresult} remains true if the kernel $k$ has a finite second moment and is merely
assumed to be non-increasing for all 
sufficiently large $j\in \iN$, see Remark \ref{extensionposresult}. This allows to cover also kernels, e.g.,
of the
form \eqref{kbeispiel} where $c,\alpha,\delta>0$ and $\gamma< 0$.

It is still an interesting open problem whether $CD(0,d)$ holds for some finite $d>0$ in case of the algebraic kernel with
critical value $\beta=2$, that is, $k(j)=c|j|^{-3}$, $j\in \iZ\setminus\{0\}$, with $c>0$. 
{We can
immediately deduce from Theorem \ref{posresult} that} this is indeed the case if the kernel has an
additional logarithmic factor, more precisely, if $k$ is of the form 
\[
k(j)=\,\frac{c}{|j|^{3}[1+\log(|j|)]^{1+\alpha}},\quad j\in \iZ\setminus \{0\},
\]
with $c,\alpha>0$.

Another striking phenomenon is that all kernels with unbounded support, even those which are bounded above by an exponential kernel, fail
the $CD(0,d)$-condition for all finite $d>0$ if the support enjoys a certain number theoretic property,
which in some sense means that {\em the support is sufficiently thin}.
The latter holds in particular if the gaps between the (positive) elements of the support grow faster than the sequence of the powers of three, see Theorem \ref{thinsupport}. Somewhat surprisingly, if the support is precisely the set $\{\pm 3^l:\,l\in \iN\}$ and the kernel
decays exponentially then $CD(0,d)$ is still valid for some finite $d>0$ as we show in Theorem \ref{powersof3}.
Thus the latter situation constitutes an extreme case. 

Having summarized the main results on the $CD(0,d)$-condition, we now come to the question whether a {\em positive
curvature} bound is possible, that is $CD(\kappa,\infty)$ for some $\kappa>0$. In view of the representation formulas \eqref{Gamma1Formel} and \eqref{Gamma2Formel}
this cannot be expected, since for the function $u(j)=j$ one obtains (at least formally) that $\Gamma_2(u)(0)=0$ and
$\Gamma(u)(0)=\sum_{j\in \iN} k(j) j^2\in (0,\infty]$. For kernels with finite second moment, this formal
argument becomes rigorous as $u(j)=j$ is an admissible function in the sense that $\Gamma(u)$ is finite. In the other case,
that is, $k$ has no finite second moment, we are able to give a rigorous argument assuming that $k$ is non-increasing
on $\iN$. By means of suitable approximation, we can here even restrict the class of admissible functions to compactly
supported functions $u$, see Theorem \ref{theo:nocurvature}. Again, the argument is rather technical due to the necessity of distinguishing several cases
when estimating the $\Gamma_2$-term. The general case (without monotonicity assumption on $k$) remains open.

Note, however, that $CD(\kappa,\infty)$ with some $\kappa>0$ implies $CD(0,d)$ for some finite $d>0$, by 
H\"older's inequality (see also Lemma \ref{lem:NoPositiveCurvature} below). Consequently, only kernels with $CD(0,d)$ for
some $d\in (0,\infty)$ come into consideration with regard to a possible positive curvature bound. Interestingly, all such kernels we know (compare the above mentioned results) do have a finite second moment! {
This raises the question }whether a finite second 
moment is necessary for $CD(0,d)$ with some $d\in (0,\infty)$. This open question brings us finally back to the open problem about the algebraic kernel in the critical case.

{The article is organised as follows. In the next section we present some preliminary results and treat the case of bounded support of $k$. In Section \ref{sec:betaKleiner2} we show that algebraic kernels with $\beta < 2$ do not satisfy $CD(0,d)$ with finite $d$ by constructing a compactly supported counterexample.
The positive result for kernels with finite second moment, which in particular applies to algebraic kernels with $\beta > 2$, is shown in the subsequent section. Section \ref{Sect:Sparse} is devoted to the combinatorial analysis of kernels with sparse support, and in the last section we show that positive curvature is not possible for monotone kernels.
}

\subsection*{Acknowledgement} 
Adrian Spener and Rico Zacher are supported by the DFG (project number 355354916, GZ ZA 547/4-1). The authors thank the anonymous referee for the valuable comments, in particular the interesting observation described in Remark \ref{rem:infiniteproduct}.

\section{Basic properties and kernels with finite support}

In this section we collect some preliminary results and show that finitely supported kernels always have finite dimension,
that is $CD(0,d)$ holds for some finite $d>0$. Moreover, we obtain some $CD(\kappa,d)$-inequalities with finite $d>0$ but negative curvature $\kappa$ by adapting previous results on graphs. Recall that we always assume that the kernel $k$ is subject to the conditions (K1) and (K2) with $|k|_1>0$. 

In order to describe the class of admissible functions, we introduce, for $1\le p<\infty$ the weighted $l_p$-spaces
\[
l_{p,k}(\iZ):=\{v:\iZ\to \iR:\, \sum_{j\in \iZ} k(j)|v(j)|^p<\infty\},
\]
endowed with the canonical norm. In what follows we use the following convention with respect to the admissible class of functions for the  
$CD(\kappa,d)$-inequality at $x\in \iZ$. If not stated otherwise, a function $u:\iZ\to \iR$ is admissible for $CD(\kappa,d)$
at $x$
if $u(\cdot+x)\in l_{1,k}(\iZ)$ in case $\kappa=0$ and otherwise $u(\cdot+x)\in l_{2,k}(\iZ)$. This ensures finiteness of all the terms appearing on the right of \eqref{CDBed} when used in the form given by \eqref{Gamma1Formel} and \eqref{Gamma2Formel}, respectively.
Observe that $l_{2,k}(\iZ)\hookrightarrow l_{1,k}(\iZ)$, since $k\in l_1(\iZ)$ and by H\"older's inequality. Bounded functions
are always admissible. Note that, by Lebesgue's theorem, $CD(\kappa, d)$ holds for all bounded functions provided this condition is satisfied for all compactly supported functions, but it seems to be not clear whether the latter implies $CD(\kappa,d)$ for all admissible functions.

The following auxiliary result will be frequently used in the paper.
\begin{lemma}\label{lem:symmetry}
When studying the validity of the $CD(\kappa,d)$-condition at $x\in \iZ$, we may always assume without loss of generality that $x=0$ and that the admissible functions satisfy $u(0) = 0$.\\
Moreover, if $CD(-\kappa,d)$ holds for all \emph{symmetric} $u$, where  $\kappa \geq 0$ and $d\in (0,\infty)$, then $CD(-\kappa,d)$ holds for all $u$.
\end{lemma}
\begin{proof}
For the first assertion we let $x \in \iZ$ and $u$ be an arbitrary admissible function. Setting $\tilde u(y) := u(y+x) - u(x)$ 
we find that $\tilde u$ is admissible at $0$ and 
 \begin{align*}
  4\Gamma_2(u)(x) &=  \sum_{j,l\in \iZ} k(j)k(l) \left(u(x+j+l)-u(x+j) - u(x+l) + u(x)\right)^2\\
  &=\sum_{j,l\in \iZ} k(j)k(l) \left( [u(x+j+l)-u(x)] - [u(x+j)-u(x)] - [u(x+l)-u(x)]\right)^2\\
  &=4\Gamma_2(\tilde u )(0).
 \end{align*}
 Similarly we have
$\mathcal{L} u(x) = \mathcal{L} \tilde u(0)$ and $\Gamma(u)(x) = \Gamma(\tilde u)(0)$.

To show the second part we may apply the first part and assume w.l.o.g.\ that $CD(-\kappa,d)$ holds at $0$ for all symmetric functions and $u$ is not symmetric with $u(0)=0$. We find for the symmetric function $ \tilde u(x) = \frac{1}{2}(u(x) + u(-x))$ that $\Gamma_2(\tilde u)(0) \leq \Gamma_2({u})(0)$, since by the inequality $(a+b)^2\le 2 a^2+2b^2$, $a,b\in \iR$, we have
\begin{align*}
(\tilde u(l+j)-\tilde u(l)-\tilde u(j))^2 &=\frac{1}{4}((u(l+j) -u(l)-u(j)) + (u(-l-j)-u(-l)-u(-j))]^2 \\
& \leq \frac{1}{2}[((u(l+j) -u(l)-u(j))^2 + (u(-l-j)-u(-l)-u(-j))^2],
\end{align*}
and hence
\[
\Gamma_2(\tilde{u})(0)\le \frac{1}{2}\,\big(\Gamma_2(u)(0)+\Gamma_2(u(-\cdot))(0)\big)=\Gamma_2(u)(0)
\]
by symmetry of $k$.
Moreover we have $\mathcal{L} u (0) = \mathcal{L} \tilde u(0)$. 
By the same reasoning as before we have $\Gamma(\tilde u)(0) \leq \Gamma(u)(0)$, hence $-\kappa\Gamma(\tilde u)(0) \geq -\kappa\Gamma(u)(0)$ and whence by the assumption
\[
\Gamma_2( u)(0) \geq \Gamma_2(\tilde u)(0) \geq -\kappa \Gamma(\tilde u)(0) + \frac{1}{d}(\mathcal{L} \tilde  u(0))^2 \geq -\kappa \Gamma( u)(0) + \frac{1}{d}( \mathcal{L} u(0))^2 .\qedhere 
\]\end{proof}

Under symmetry of $u$ and assuming $u(0)=0$ we have (by taking $l=-j$ in formula \eqref{Gamma2Formel} for $\Gamma_2$) the basic estimate
\begin{equation} \label{basicGamma2Est}
\Gamma_2(u)(0)\ge \frac{1}{4}\,\sum_{j\in \iZ} k(j)^2 \big(u(0)-u(j)-u(-j)\big)^2=2\sum_{j=1}^\infty k(j)^2 u(j)^2.
 \end{equation}
 Note that \eqref{basicGamma2Est} is not sufficient to show a $CD(0,d)$-inequality with finite dimension $d$ if $k$ has unbounded support. Indeed, letting $u(j)= \frac{1}{k(j)}$ for $j \in \operatorname{supp} k$ with $|j| \leq N$ and zero otherwise we obtain
 \[\sum_{j=1}^\infty k(j)^2 u(j)^2 = \#\{ j \in \operatorname{supp} k \,:\,1\le  j \leq N\} 
 \]
whereas the right hand side satisfies
\[
 (\mathcal{L}u(0))^2 = \left(2 \sum_{j=1}^\infty k(j) u(j)\right)^2= 4(\#\{ j \in \operatorname{supp} k \,:\,1\le  j \leq N\})^2 
\]
and grows faster as $N\to \infty$. This observation is a basic ingredient for the results in Section \ref{Sect:Sparse}. In the case of a finitely supported $k$ the lower bound \eqref{basicGamma2Est} is enough to obtain the following important result.
\begin{satz} \label{finite support}
If the kernel $k$ has finite support, i.e. $ \#\operatorname{supp} k =2N$ for some $N \in \N$, then the operator $\cL$ satisfies $CD(0,2N)$. 
\end{satz}
\begin{proof}
By Lemma \ref{lem:symmetry}, it suffices to prove the asserted CD-inequality at the point $x=0$ for all (admissible) symmetric
functions $u$ with $u(0)=0$.
By (K1) and (K2) we may assume that 
$\operatorname{supp} k = \{-x_N, \ldots, -x_1, x_1, \ldots, x_N\} \subset \iZ\setminus\{0\}$. 
For simplicity of notation we set $x_0 := 0$ and let $u$ be a symmetric admissible function with $u(0) = 0$. Taking $l = -j$  (and hence $x_{-j} = -x_j$) in the sum below we obtain that
\begin{align*}
  \Gamma_2(u)(0)& = \frac{1}{4} \sum_{j=-N}^N \sum_{l=-N}^N k(x_j)k(x_l) \left(u(x_j+x_l)-u(x_j) - u(x_l)\right)^2
  \\& \geq 
  \frac{1}{4} \sum_{j=-N}^N k(x_j)k(-x_j) \left(u(x_j+(-x_j))-u(x_j) - u(-x_j)\right)^2\\
  &= 2 \sum_{j = 1}^N (k(x_j) u(x_j))^2 
  \geq \frac{2}{N}  \left( \sum_{j = 1}^N k(x_j) u(x_j)\right)^2
  \\&= \frac{1}{2N} \left( \sum_{j = -N}^N k(x_j) u(x_j)\right)^2 = \frac{1}{2N}(\cL u(0))^2
\end{align*}
after employing with $y_i := k(x_i)u(x_i)$ the Cauchy-Schwarz inequality $\left(\sum _{i=1}^N y_i \right)^2 \leq N \sum _{i=1}^N y_i^2$. \qedhere
\end{proof}

The following example shows that Theorem \ref{finite support} is optimal
in general.
\begin{bei}
{\em 
For $N \in \N$ we define
\[
k(x) = \begin{cases}
1, & x = 2j+1 \text{ for some } -N \leq j\leq N-1\\
0,& \text{else}
\end{cases}
\quad \text{ and } \quad
u(x) = \begin{cases}
1,& x \text{ is odd}\\
0, & x = 0\\
2, &\text{else.}
\end{cases}
\]
Then $k$ is symmetric and satisfies the assumption of Theorem
\ref{finite support}, and we calculate
\[
\mathcal{L}(u)(0) = \sum_{j=-N}^{N-1} k(2j+1)u(2j+1) =
\sum_{j=-N}^{N-1} 1 = 2N.
\]
For $\Gamma_2$ we find that $u(2j + 2l +2) \neq 2$ if and only if $l =
-j-1$, whence
\[
\Gamma_2(u)(0) = \frac{1}{4}
\sum_{j=-N}^{N-1} \sum_{l=-N}^{N-1} \big(u(2l+2j+2) - 1 -1 \big)^2 =
\frac{1}{4} \sum_{j=-N}^{N-1} 4 = 2N = \frac{1}{2N}
\big(\mathcal{L}(u)(0)\big)^2.\qedhere
\]
}
\end{bei}

Clearly, if $CD(\kappa, d)$ is satisfied with some positive curvature $\kappa$, then $CD(0,d)$ holds. Similarly we find that $CD(\kappa,\infty)$ with $\kappa>0$ implies $CD(0,d)$ for some finite $d$.
\begin{lemma}
\label{lem:NoPositiveCurvature}
 If $CD(\kappa,d)$ holds for some $\kappa \geq 0$ and $d \in (0,\infty]$, then $CD(0,\tilde d)$ holds, where $\tilde d =\frac{2d|k|_1}{d \kappa + 2|k|_1}$ for $d < \infty$ and $\tilde d =\frac{2|k|_1}{\kappa}$ if $d = \infty$.
\end{lemma}
\begin{proof}
 Again, we may restrict ourselves to the described CD-inequalities at $x=0$ and we may assume that
 the admissible functions vanish at zero.
   
   First we observe by applying Jensen's inequality that
 $ (\mathcal{L} u(0))^2 \leq 2 |k|_1  \Gamma(u)(0)$.
Indeed, we have
\[
(\mathcal{L} u(0))^2  
=|k|_1^2\left( \sum_{j\in \iZ} \frac{k(j)}{|k|_1} u(j)\right)^2
\leq |k|_1 \sum_{j\in \iZ} k(j) u(j)^2
= 2|k|_1 \Gamma(u)(0),
\]
whence \[
\Gamma_2(u)(0) \geq \kappa \Gamma(u)(0) + \frac{1}{d} (\mathcal{L} u(0))^2
 \geq \left(\frac{\kappa}{2|k|_1} + \frac{1}{d}\right)(\mathcal{L} u(0))^2
  = \frac{1}{\tilde d}(\mathcal{L} u(0))^2
\]
where $\tilde d = \frac{2d|k|_1}{d\kappa + 2|k|_1}$.
\end{proof}

Let us present a few adaptions of general results concerning CD-inequalities on graphs.

\begin{proposition}\label{prop:CD(-k,2)}
 For any kernel $k$ we have $CD(-|k|_1,2)$.
\end{proposition}

This proposition is an immediate consequence of the following identity, which can be obtained analogously to \cite[Thm 1.3]{MR2644381}.

\begin{lemma}[{\cite[(2.9)]{MR3164168}}]
For any kernel we have 
 \begin{equation} \label{eq:LemmaLY}
  \Gamma_2(u)(0) = \frac{1}{4} \sum_{j,l} k(l)k(j) [u(j+l)-2u(j) + u(0)]^2 - |k|_1\Gamma(u)(0) + \frac{1}{2}(\mathcal{L} u(0))^2.
 \end{equation}
\end{lemma}

\begin{corollary}{\cite[Thm 1.2]{MR2644381}}
Assume that $k(j) \geq c$ for $j \in \operatorname{supp} k$. Then $CD(2c-|k|_1, 2)$ holds.
\end{corollary}

\begin{proof}
Choosing $l = -j$ in \eqref{eq:LemmaLY} we obtain
\begin{align*}
 \Gamma_2(u)(0)  &\geq \frac{4}{4} \sum_j k(j)^2 [u(0)-u(j)]^2 - |k|_1\Gamma(u)(0)  + \frac{1}{2}(\mathcal{L}(u)(0))^2
  \\&\geq (2c-|k|_1)\Gamma(u)(0)  + \frac{1}{2}(\mathcal{L}(u)(0))^2. \qedhere
\end{align*}
\end{proof}
Since $k$ is symmetric we have $2c  \leq |k|_1$ with equality only for kernels whose support consist of exactly two elements, e.g. for the discrete Laplacian \eqref{eq:DiscreteLaplacian}.

The next remark shows that for kernels with finite second moment we always have the necessary condition $d \geq 1$ if $CD(0,d)$ holds. 

\begin{bemerk1} \label{square}
{\em Suppose that the kernel is such that $\sum_{j\in \iN}j^2 k(j)<\infty$. For $u(j)=j^2$, $j\in \iZ$, we find that
 \begin{align*}
 \Gamma_2(u)(0)& =\,\frac{1}{4}\,\sum_{j,l\in \iZ} k(j)k(l) \big((j+l)^2-j^2-l^2\big)^2
 =\sum_{j,l\in \iZ} k(j)k(l) j^2 l^2\\
& = \big(\cL(u)(0)\big)^2.
 \end{align*}
Hence, if the kernel satisfies $CD(0,d)$, then $d\ge 1$. } 
\end{bemerk1}


\section{Kernels of power type with $\beta<2$}
\label{sec:betaKleiner2}
In this section we consider an important class of kernels for which $CD(0,d)$ fails to hold for all finite $d>0$. 

We use the following notation. For functions $f,g:D \to \R$ we write 
$f(x) \lesssim g(x)$ if there exists some constant $C > 0$ such that $f(x) \leq C g(x)$ for all $x \in D$. Moreover, if $f(x) \lesssim g(x) \lesssim f(x)$ we write $f(x) \sim g(x)$.
\begin{satz} \label{betalessthan2}
Let $k$ be a power type kernel as in \eqref{powerkerneldef} with $c>0$ and $\beta\in (0,2)$.
Then the corresponding operator $\cL$ fails to satisfy $CD(0,d)$ for all finite $d>0$.
\end{satz}
\begin{proof}
We may assume without restriction of generality that $c=1$. 
We consider the family of functions
\[
u_\varepsilon(j)=|j|^{\beta-\varepsilon},\quad j\in \iZ, \;\varepsilon\in (0,\beta).
\]
We will show that 
\begin{equation} \label{ideacounterex}
\cL(u_\varepsilon)(0)\sim \frac{1}{\varepsilon}\;\mbox{and}\;\Gamma_2(u_\varepsilon)(0) \sim \frac{1}{\varepsilon}\quad \mbox{as}\;\varepsilon\to 0,
\end{equation}
which implies that for small $\varepsilon>0$ 
\[
\Gamma_2(u_\varepsilon)(0)\le C \varepsilon (\cL(u_\varepsilon)(0))^2
\]
with some constant $C>0$, contradicting any $CD(0,d)$ inequality with finite $d>0$. Observe that \eqref{ideacounterex}
contains more information than what is actually required for the proof of Theorem
\ref{betalessthan2}. In fact, concerning $\Gamma_2$ it would be enough to show that $\Gamma_2(u_\varepsilon)(0) \lesssim \frac{1}{\varepsilon}$ as $\varepsilon\to 0$.

The first claim in \eqref{ideacounterex} can be easily verified. Indeed, 
\[
\cL(u_\varepsilon)(0)=2 \sum_{j=1}^\infty k(j)u_\varepsilon(j)=2 \sum_{j=1}^\infty \frac{1}{j^{1+\beta}} 
j^{\beta-\varepsilon}=2 \sum_{j=1}^{\infty} \frac{1}{j^{1+\varepsilon}} \sim  \frac{1}{\varepsilon}
\]
as $\varepsilon\to 0$, since {by Lemma \ref{lem:IntegralComparison}} the last sum can be controlled from below and above by a positive constant times the integral
\[
\int_1^\infty \frac{dx}{x^{1+\varepsilon}}=\,\frac{1}{\varepsilon}.
\]

Turning to $\Gamma_2$, by symmetry it is enough to consider the terms 
\[
J_\varepsilon:=\,\sum_{j=1}^\infty \sum_{l=1}^{j}k(j)k(l)\big(u_\varepsilon(j+l)-u_\varepsilon(j)-u_\varepsilon(l)\big)^2
\]
and 
\[
K_\varepsilon:=\,\sum_{j=1}^\infty \sum_{l=1}^{j}k(j)k(l)\big(u_\varepsilon(j-l)-u_\varepsilon(j)-u_\varepsilon(l)\big)^2.
\]
The first term can be reformulated as follows.
\begin{align}
J_\varepsilon= & \,\sum_{j=1}^\infty \sum_{l=1}^{j}\frac{1}{j^{1+\beta}l^{1+\beta}}\big((j+l)^{\beta-\varepsilon}
-j^{\beta-\varepsilon}-l^{\beta-\varepsilon} \big)^2\nonumber\\
= & \,\sum_{j=1}^\infty \sum_{l=1}^{j}\frac{j^{2\beta-2\varepsilon}}{j^{1+\beta}l^{1+\beta}}
\big[\Phi_\gamma\big(\frac{l}{j}\big)\big]^2, \label{Jeps1}
\end{align}
where $\gamma=\beta-\varepsilon$ and the function $\Phi_\gamma$ is defined by
\[
\Phi_\gamma(x)=(1+x)^\gamma-1-x^\gamma,\quad x\in [0,1].
\]
\begin{lemma} \label{philemma}
In the case $\gamma\in (1,2)$ there holds
\[
(\gamma-1) x\le \Phi_\gamma(x)\le 2\gamma x,\quad x\in [0,1].
\]
If $\gamma\in (0,1)$, we have
\[
-x^\gamma \le \Phi_\gamma(x)\le -(1-\gamma)x^\gamma,\quad x\in [0,1].
\]
\end{lemma}
\begin{proof}
For $x\in [0,1]$ we have by the mean value theorem that $(1+x)^\gamma-1=\gamma \xi^{\gamma-1}x$
for some $\xi\in [1,1+x]$. Consequently, if $\gamma\in (1,2)$, we see that
\[
\gamma x\le (1+x)^\gamma-1\le 2\gamma x,
\]
which together with $x^\gamma\le x$ implies the first assertion. In the case $\gamma\in (0,1)$, we have for $x\in (0,1]$
that $(1+x)^\gamma-1\le \gamma x^{\gamma-1}x=\gamma x^\gamma$, since $x\le \xi$. This shows the second assertion.
\end{proof}
In order to estimate $J_\varepsilon$ by the aid of Lemma \ref{philemma} we distinguish two cases.

{\em Case 1:} Suppose that $\beta\in (1,2)$. Then $\gamma\in (1,2)$ for sufficiently small $\varepsilon$, and Lemma
 \ref{philemma} and \eqref{Jeps1} then show that $J_\varepsilon$ can be estimated from below and above as
\begin{equation} \label{Jepsbeh}
J_\varepsilon \sim C(\beta)\sum_{j=1}^\infty \sum_{l=1}^{j}\frac{j^{2\beta-2\varepsilon}}{j^{1+\beta}l^{1+\beta}}
\big(\frac{l}{j}\big)^2= C(\beta)\sum_{j=1}^\infty \frac{1}{j^{3-\beta+2\varepsilon}} \sum_{l=1}^{j}l^{1-\beta}.
\end{equation}
By monotonicity of the sequence $(l^{1-\beta})_{l\in \iN}$, the sum $\sum_{l=1}^{j}l^{1-\beta}$ can be controlled from below and
above by a positive constant times the integral
\[
\int_1^{j+1} x^{1-\beta}dx = \frac{1}{2-\beta}\,((j+1)^{2-\beta}-1)\sim j^{2-\beta}.
\]
Therefore
\[
J_\varepsilon \sim C(\beta)\sum_{j=1}^\infty \frac{1}{j^{3-\beta+2\varepsilon}}\, j^{2-\beta}=
 C(\beta)\sum_{j=1}^\infty \frac{1}{j^{1+2\varepsilon}}\sim \frac{1}{\varepsilon}
\]
as $\varepsilon\to 0$.

{\em Case 2:} Suppose now that $\beta\in (0,1]$. Then $\gamma\in (0,1)$, and by Lemma
 \ref{philemma} it follows that $J_\varepsilon$ can be estimated from below and above as
\[
J_\varepsilon \sim C(\beta)\sum_{j=1}^\infty \sum_{l=1}^{j}\frac{j^{2\beta-2\varepsilon}}{j^{1+\beta}l^{1+\beta}}
\big(\frac{l}{j}\big)^{2\beta-2\varepsilon}= C(\beta)\sum_{j=1}^\infty \frac{1}{j^{1+\beta}} 
\sum_{l=1}^{j} l^{\beta-1-2\varepsilon}.
\]
Arguing as in the first case, we see that $\sum_{l=1}^{j} l^{\beta-1-2\varepsilon}\sim j^{\beta-2\varepsilon}$ and
hence
\[
J_\varepsilon \sim C(\beta)\sum_{j=1}^\infty \frac{1}{j^{1+\beta}}\, j^{\beta-2\varepsilon}=
C(\beta)\sum_{j=1}^\infty \frac{1}{j^{1+2\varepsilon}}\sim \frac{1}{\varepsilon}
\]
as $\varepsilon\to 0$.

We now come to the term $K_\varepsilon$. We have
\begin{align}
K_\varepsilon= & \,\sum_{j=1}^\infty \sum_{l=1}^{j}\frac{1}{j^{1+\beta}l^{1+\beta}}\big(
j^{\beta-\varepsilon}+l^{\beta-\varepsilon}-(j-l)^{\beta-\varepsilon}
 \big)^2\nonumber\\
= & \,\sum_{j=1}^\infty \sum_{l=1}^{j}\frac{j^{2\beta-2\varepsilon}}{j^{1+\beta}l^{1+\beta}}
\big[\Psi_\gamma\big(\frac{l}{j}\big)\big]^2, \label{Jeps2}
\end{align}
where again $\gamma=\beta-\varepsilon$ and the function $\Psi_\gamma$ is defined by
\[
\Psi_\gamma(x)=1+x^\gamma-(1-x)^\gamma,\quad x\in [0,1].
\]
\begin{lemma} \label{psilemma}
In the case $\gamma\in (1,2)$ there holds
\[
 x\le \Psi_\gamma(x)\le (\gamma+1) x,\quad x\in [0,1].
\]
If $\gamma\in (0,1)$, we have
\[
 x^\gamma \le \Psi_\gamma(x)\le 3 x^\gamma,\quad x\in [0,1].
\]
\end{lemma}
\begin{proof}
Suppose first that $\gamma\in (1,2)$. For $x\in [0,1]$, we have
$1-(1-x)^\gamma=\gamma \xi^{\gamma-1} x$ for some $\xi\in [1-x,1]$, by the mean value theorem.
Since $\gamma-1>0$, this implies $1-(1-x)^\gamma\le \gamma x$, which together with $x^\gamma\le x$ yields
the upper bound in the assertion. For $x\in [0,\frac{1}{2}]$ we have $\xi\in [\frac{1}{2},1]$, which implies that
$\Psi_\gamma(x)\ge 1-(1-x)^\gamma\ge \gamma (1/2)^{\gamma-1}x\ge x$. If $x\in [\frac{1}{2},1]$, then clearly
$\Psi_\gamma(x)\ge 1\ge x$. This proves the lower bound in the first claim.

Let now $\gamma\in (0,1)$. The lower estimate for $\Psi_\gamma(x)$ is evident, since $1-(1-x)^\gamma\ge 0$. For the upper bound we again use the mean value theorem similarly as above to see that for $x\in [0,\frac{1}{2}]$ there holds
$1-(1-x)^\gamma\le \gamma (1/2)^{\gamma-1}x$, noting that now $\gamma-1<0$. Since $x=x^{1-\gamma}x^{\gamma}
\le (1/2)^{1-\gamma}x^\gamma$, it follows that $\Psi_\gamma(x)\le 2 x^\gamma$. For $x\in [\frac{1}{2},1]$ we
have $1\le 2x\le 2x^\gamma$ and hence $\Psi_\gamma(x)\le 1+x^\gamma \le 3 x^\gamma$. This shows the upper bound
in the second assertion.
\end{proof}

Having Lemma \ref{psilemma} at disposal, we may now estimate $K_\varepsilon$ appropriately. As before, we distinguish
two cases w.r.t.\ the parameter $\beta$.

{\em Case 1:} Assume that $\beta\in (1,2)$. Then \eqref{Jeps2}, the first part of Lemma \ref{psilemma} {and Lemma \ref{lem:IntegralComparison}} show that for sufficiently small $\varepsilon>0$, $K_\varepsilon$ can be estimated from below and above as
\[
K_\varepsilon \sim C(\beta) \,\sum_{j=1}^\infty \sum_{l=1}^{j}\frac{j^{2\beta-2\varepsilon}}{j^{1+\beta}l^{1+\beta}}
\big(\frac{l}{j} \big)^2,
\]
which is the same expression as in the estimation of $J_\varepsilon$ in case 1. Therefore, $ K_\varepsilon \sim \frac{1}{\varepsilon}$ as $\varepsilon\to 0$.

{\em Case 2:} Suppose now that $\beta\in (0,1]$. Then \eqref{Jeps2}, the second part of Lemma \ref{psilemma} {and Lemma \ref{lem:IntegralComparison}} show that for sufficiently small $\varepsilon>0$, $K_\varepsilon$ can be controlled from below and above as
\[
K_\varepsilon \sim C \,\sum_{j=1}^\infty \sum_{l=1}^{j}\frac{j^{2\beta-2\varepsilon}}{j^{1+\beta}l^{1+\beta}}
\big(\frac{l}{j} \big)^{2\beta-2\varepsilon}.
\]
This is the same expression as in the estimation of $J_\varepsilon$ in case 2. Hence $ K_\varepsilon \sim \frac{1}{\varepsilon}$ as $\varepsilon\to 0$.
All in all, we see that $\Gamma_2(u_\varepsilon)(0) \sim \frac{1}{\varepsilon}$ as $\varepsilon\to 0$. 
The proof of Theorem \ref{betalessthan2} is complete.
\end{proof}

Interestingly, in the case $\beta=2$ the family $u_\varepsilon$ considered above does no longer lead to a contradiction of 
$CD(0,d)$ for finite $d>0$. This is due to the fact that now $\Gamma_2(u_\varepsilon)(0) \sim \frac{1}{\varepsilon^2}$ as
$\varepsilon\to 0$. To see this, we can again use the statements in Lemma \ref{philemma} and \ref{psilemma}
where $\gamma\in (1,2)$. As to the term $J_\varepsilon$, we see from \eqref{Jepsbeh} with $\beta=2$ that 
\[
J_\varepsilon \sim C\sum_{j=1}^\infty \frac{1}{j^{1+2\varepsilon}} \sum_{l=1}^{j}l^{-1}. 
\]   
In contrast to the case $\beta<2$, the inner sum growths logarithmically in $j$ and thus
\[
J_\varepsilon \sim C\sum_{j=1}^\infty \frac{1}{j^{1+2\varepsilon}}(1+\log j)\sim \frac{1}{\varepsilon^2},
\]
since $\int_c^\infty \frac{\log x}{x^{1+2\varepsilon}}\,dx \sim  \frac{1}{\varepsilon^2}$ as $\varepsilon\to 0$, for all $c\ge 1$. The term $K_\varepsilon$ enjoys the same behaviour. 


As we have mentioned in the introduction, an important consequence of Theorem \ref{betalessthan2} is the following corollary concerning fractional powers of the discrete Laplacian.
\begin{korollar}
The operator $- \big( - \Delta^{\frac{\beta}{2}} \big)$ fails to satisfy $CD(0,d)$ for all finite $d>0$.
\end{korollar}
\begin{proof}
Let $\beta \in (0,2)$ be fixed. From \cite[Theorem 1.1]{MR3787555} we obtain that $- \big( - \Delta^{\frac{\beta}{2}}\big)$ is an operator of the form \eqref{genlaplacedef}, where the kernel is given by
\begin{equation}\label{eq:stingakernel}
k_\beta (j) = \frac{4^{\frac{\beta}{2}} \Gamma\left(\frac{1+\beta}{2}\right)\Gamma \left( | j | - \frac{\beta}{2}\right)}{\sqrt{\pi} \Gamma \left(- \frac{\beta}{2}\right)\Gamma \left( | j | + 1 + \frac{\beta}{2}\right)}, \quad  j \in \mathbb{Z}\setminus \{0\},
\end{equation}
and $k_\beta(0)=0$. Note that in \eqref{eq:stingakernel}, $\Gamma$ denotes  the Gamma function. Throughout this proof we denote by $\mathcal{L}$ and $\Gamma_2$ the operators corresponding to the respective power type kernel from \eqref{powerkerneldef} and by $\Gamma_2^{(\beta)}$ the iterated carr\'e du champ operator for $- \big( - \Delta\big)^\frac{\beta}{2}$.
In \cite[Theorem 1.1]{MR3787555} it is shown that one finds constants $0 < c(\beta) \leq C(\beta)$ such that 
\begin{align*}
\frac{c(\beta)}{|j|^{1+\beta}} \leq k_\beta(j) \leq \frac{C(\beta)}{|j|^{1+\beta}}
\end{align*}
holds for any $j \in \mathbb{Z}\setminus\{0\}$. Thus, choosing the non-negative function $u_\varepsilon$ from the proof of Theorem \ref{betalessthan2} we observe $\big( - \big( - \Delta\big)^{\frac{\beta}{2}} (u_\varepsilon) (0)\big)^2 \geq c(\beta)^2 \left(\mathcal{L}u_\varepsilon (0) \right)^2$ and $\Gamma_2^{(\beta)}(u_\varepsilon)(0) \leq C(\beta)^2\; \Gamma_2 (u_\varepsilon) (0)$. The claim follows from \eqref{ideacounterex}.
\end{proof}

A natural question to ask is whether the statement of Theorem \ref{betalessthan2} remains true for a smaller class of functions, e.g. the space of bounded functions, since the above counterexample is obviously unbounded. We can answer this question negatively by an approximation argument. 

\begin{satz}\label{theo:approximation}
Let $k$ be a power type kernel as in \eqref{powerkerneldef} with $c>0$ and $\beta\in (0,2)$. There exists no finite $d>0$ such that 
\begin{align*}
\Gamma_2 (v) (0) \geq \frac{1}{d} \left(\left( \mathcal{L} v  \right)(0)\right)^2
\end{align*}
holds for all compactly supported functions $v$.
\end{satz}

For the proof of this statement we will repeatedly use the following lemma.

\begin{lemma}\label{lem:IntegralComparison}
Let $A_1, A_2 \in \N$ with $A_1 < A_2$ and $\gamma \in \mathbb{R}$. Then the following estimates are valid
\begin{equation}\label{eq:intvergleichlarge}
\sum\limits_{m=A_1}^{A_2} m^\gamma \leq \begin{cases} \frac{1}{\gamma + 1} \left(A_2 + 1 \right)^{1+\gamma}, &1+\gamma \geq 1 \\
\frac{1}{\gamma + 1} A_2^{1+\gamma}, &1+\gamma \in (0,1),
\end{cases}
\end{equation}
and if $A_1 \geq 2$
\begin{equation}\label{eq:intvergleichlow}
\sum\limits_{m=A_1}^{A_2} m^\gamma \leq \begin{cases}
\log(A_2), &1+\gamma = 0 \\
\frac{1}{|\gamma + 1|} \left(A_1 - 1 \right)^{1+\gamma}, &1+\gamma <0.
\end{cases}
\end{equation}
\end{lemma}

\begin{proof}
For $\gamma \geq 0$ we observe
\begin{align*}
\sum\limits_{m=A_1}^{A_2} m^\gamma \leq \sum\limits_{m=A_1}^{A_2} \int\limits_{m}^{m+1} x^\gamma \; \mathrm{d} x 
=\frac{1}{\gamma +1} \left( (A_2 +1)^{\gamma + 1} - A_1^{\gamma + 1} \right) \leq \frac{1}{\gamma + 1} \left(A_2 + 1 \right)^{\gamma +1}.
\end{align*}
In case of $\gamma < 0$ we have
\begin{align*}
\sum\limits_{m=A_1}^{A_2} m^\gamma \leq \sum\limits_{m=A_1}^{A_2}\; \int\limits_{m-1}^{m} x^\gamma \; \mathrm{d} x,
\end{align*}
since the mapping $x \mapsto x^\gamma$ is now decreasing for $x>0$. Assuming that $A_1 \geq 2$, we can proceed calculating the integral as above and obtain
\begin{equation*}
\sum\limits_{m=A_1}^{A_2} m^\gamma \leq \begin{cases} 
\frac{1}{\gamma+1} \left( A_2^{\gamma +1} - (A_1-1)^{\gamma + 1} \right) , & \gamma \in (-\infty,0)\setminus \{-1\} \\
\log(A_2) - \log(A_1-1) , & \gamma=-1.
\end{cases}
\end{equation*}
From this the claim follows.
\end{proof}

\begin{proof}[Proof of Theorem \ref{theo:approximation}]
Let $\beta \in (0,2)$ be fixed and choose $\varepsilon > 0$ such that $\beta - 2\varepsilon > 0$ and $\beta - \varepsilon >1$ in case of $\beta >1$. Let $u_\varepsilon$ be given as in the proof of Theorem \ref{betalessthan2}.

We define for {even} $N\in 2\mathbb{N}$ 
the function
\begin{align*}
v_{N,\varepsilon}(j):= \begin{cases}
u_\varepsilon(j) ,& j\in\{0,...,N\}\\
-\frac{N^{\beta - \varepsilon}}{N^2-N} j + \frac{N^{\beta - \varepsilon +2}}{N^2 - N} ,&j \in \{N+1,...,N^2\} \\
0 , & j >N^2,
\end{cases}
\end{align*}
and extend it symmetrically to  $v_{N,\varepsilon} : \mathbb{Z} \to \R$. In the sequel we will denote  $u_\varepsilon$ by $u$ and $v_{N,\varepsilon}$ by $v_N$.

Our aim is to prove that $\mathcal{L}(v_N)(0) \to \mathcal{L}(u)(0) $ and $\Gamma_2(v_N)(0) \to \Gamma_2 (u) (0)$ as $N \to \infty$, which is  sufficient to deduce the claim. Indeed, assuming that there exists some $\Lambda>0$ such that $\Gamma_2(w) \geq \Lambda(\mathcal{L}w)^2$ at $x = 0$ for all compactly supported functions $w$, we find from Theorem \ref{betalessthan2} a sufficiently small $\varepsilon>0$ with $u$, as given above,  satisfying $\Gamma_2(u)(0) \leq \frac{\Lambda}{8} (\mathcal{L} u(0))^2$. Let $N$ be large enough such that
$|(\mathcal{L}v_N(0))^2 - (\mathcal{L}u(0))^2| \leq \frac{1}{2} (\mathcal{L}u(0))^2$ and 
$|\Gamma_2(v_N)(0) - \Gamma_2(u)(0)| \leq \Gamma_2(u)(0)$.
Then $(\mathcal{L} u(0))^2 \leq 2 (\mathcal{L}v_N(0))^2$ and whence
\[
\Gamma_2(v_N)(0) \leq 2 \Gamma_2(u)(0) \leq \frac{\Lambda}{4} (\mathcal{L} u(0))^2 \leq \frac{\Lambda}{2}  (\mathcal{L} v_N(0))^2,
\]
contradicting the assumption and showing the claim.

First we obtain that
\begin{align*}
| \mathcal{L}(u)(0) - \mathcal{L}(v_N)(0) | \leq 2 \left( \sum\limits_{j=N+1}^{N^2} \frac{| u (j) - v_N(j)|}{j^{1+\beta}} + \!\! \sum\limits_{j=N^2+1}^\infty \frac{u(j)}{j^{1+\beta}} \right) \leq 2 \sum\limits_{j=N+1}^\infty \frac{u(j)}{j^{1+\beta}}  \; \xlongrightarrow{N \to \infty} 0.
\end{align*}
Next, we fix some $\rho>0$ and aim to prove $\vert \Gamma_2(v_N)(0) - \Gamma_2(u)(0) \vert < \rho$ for each $N \geq N_0(\rho)$ for some sufficiently large $N_0(\rho)\in \mathbb{N}$. For $M\in \N$ we define the kernel  $k_M(x) := \frac{1}{|x|^{1+\beta}}
\chi_{(\{-M,...,M\}\setminus \{0\})}$ and denote the corresponding operator from \eqref{genlaplacedef} by  $\mathcal{L}^M$. Furthermore, we denote by $\Gamma_2^M$ the respective iterated carr\'e du champ operator, which can be written as
\begin{align*}
\Gamma_2^M(w)(x)=\,\frac{1}{4}\,\sum_{j,l\in \{-M,...,M\}\setminus \{0\}} \frac{\big(w(x+j+l)-w(x+j)-w(x+l)+w(x)\big)^2}{j^{1+\beta} l^{1+\beta}} ,\quad x\in \iZ.
\end{align*}
From the dominated convergence theorem one obtains that $\Gamma_2^M (u)(0)$ converges to $\Gamma_2 (u)(0)$ as $M$ tends to infinity. We fix $M>0$ large enough such that 
\begin{equation*}
|\Gamma_2^M (u) (0) - \Gamma_2 (u) (0) | < \frac{\rho}{4}.
\end{equation*}
If $N > 2M$ and $|j| , |l| \leq M$ we have $( v_N(j+l)-v_N(j) - v_N(l))^2 = ( u(j+l) - u(j) - u(l) )^2$ and thus it holds
\begin{equation*}
|\Gamma_2^M (v_N)(0) - \Gamma_2^M (u)(0)| = 0.
\end{equation*}
The above observations are beneficial, due to the basic calculation
\begin{align*}
&| \Gamma_2 (v_N)(0) - \Gamma_2 (u)(0)| \\ 
& \qquad \leq | \Gamma_2 (v_N)(0) - \Gamma_2^M (v_N)(0)| + | \Gamma_2^M (v_N)(0) - \Gamma_2^M (u)(0) | + | \Gamma_2^M (u)(0) - \Gamma_2 (u) (0) |.
\end{align*}
Hence, to show the claim it suffices to prove that $| \Gamma_2 (v_N)(0) - \Gamma_2^M (v_N)(0)|$ converges to zero as $N$ tends to infinity.

In order to prove the desired convergence, we have to distinguish several cases. 
We observe by symmetry of the kernel and $v_N$ that 
\begin{equation}\label{eq:Gammatriangle}
\begin{split}
&| \Gamma_2 (v_N)(0) - \Gamma_2^M (v_N)(0)| \\
& \qquad \lesssim \sum\limits_{j=M+1}^\infty \sum\limits_{l=1}^j \frac{\left( v_N(j+l) - v_N(j)-v_N(l)\right)^2 + \left( v_N(j-l) - v_N(j)-v_N(l)\right)^2}{j^{1+\beta}l^{1+\beta}} .
\end{split}
\end{equation}
In order to show that the expression in \eqref{eq:Gammatriangle} converges to zero we distinguish several cases (see Figure \ref{fig:domain1} and \ref{fig:domain2}). First we make the following observation. Given $1 \leq x \leq y$ one easily obtains from the mean value theorem the upper bound
\begin{equation}\label{eq:nonmixedmvt}
\vert v_N(x)-v_N(y) \vert \lesssim \begin{cases}
\max\{ x^{\beta - \varepsilon -1},y^{\beta - \varepsilon -1} \} \left( y-x \right) , &y \leq N \\
N^{\beta - \varepsilon -2} \left( y - x\right) , &x > N,
\end{cases}
\end{equation}
Here and in the following we write $\lesssim$ if the corresponding constant  is independent of $N$ and $M$. In the remaining situation $x \leq N < y$, we apply \eqref{eq:nonmixedmvt} to obtain
\begin{align*}
\vert v_N(y)- v_N(x) \vert &\leq \vert v_N(y) - v_N(N) \vert + \vert v_N(N) - v_N(x) \vert \\
&\lesssim N^{\beta - \varepsilon -2} (y-N) + \max\{N^{\beta - \varepsilon -1}, x^{\beta -\varepsilon -1} \} (N-x).
\end{align*}
If $\beta >1$ this leads to 
\begin{equation}\label{eq:mixedgeq1}
\vert v_N(y)- v_N(x) \vert \lesssim N^{\beta - \varepsilon -1} (y-x),
\end{equation}
since $N^{\beta - \varepsilon -2} \leq N^{\beta - \varepsilon -1}$. In the other case of $\beta \leq 1$, we have
\begin{equation}\label{eq:mixedleq1}
\vert v_N(y)- v_N(x) \vert \lesssim  x^{\beta - \varepsilon -1} (y-x),
\end{equation}
since $N^{\beta - \varepsilon -2} \leq \frac{1}{N}x^{\beta - \varepsilon -1} \leq x^{\beta - \varepsilon -1}$.
\begin{figure}[ht]
 \centering
\begin{tikzpicture}[scale = 1]
\footnotesize
 \draw (0,0) -- (5,5);
\draw[->,thick] (0,-0.3) -- (0,5.4);
 \draw[->, thick] (-0.3,0) -- (6.4,0);
 \draw (1,0) -- (1,1);
 \draw (1.5,1.5) -- (3,0)-- (3,3);
\node at (1.7,0.5) {I};
\node at (2.4,1.3) {II};
\node at (4,2.5) {III};
 \node  at (-0.4, 4) {$l$};
 \node  at (5, -0.4) {$j$};
 \node  at (1,-0.4) {M};
 \node  at (3,-0.4) {N};
   \node (n4c) at (-0.5,3) {N};
   \node (n4c) at (0,3) {-};
      \node (n4c) at (-0.5,1) {M};
   \node (n4c) at (0,1) {-};
\end{tikzpicture}
 \caption{Splitting of the domain of summation for the term including $j+l$.}\label{fig:domain1}
\end{figure}
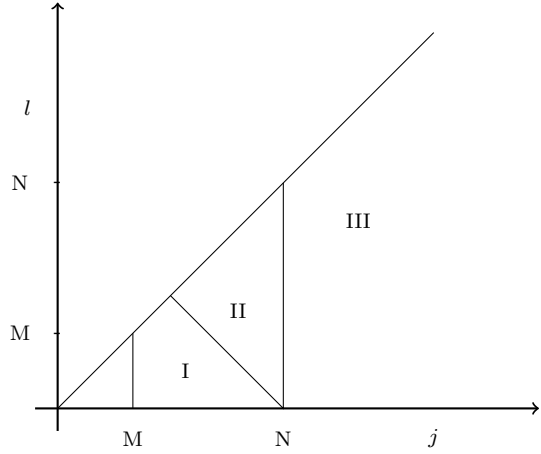\\
\underline{I: $1\leq l \leq j, j+l \leq N, j \geq M+1$:} Clearly, if $j+l \leq N$ we have $\left(v_N(j+l)-v_N(j)-v_N(l)\right)^2 = \left(u(j+l)-u(j)-u(l)\right)^2 $. We can control this case  by the sum
\begin{equation*}
\sum\limits_{j=M+1}^{N-1} \sum\limits_{l=1}^{N-j} \frac{\left( v_N(j+l) - v_N(j)-v_N(l)\right)^2}{j^{1+\beta}l^{1+\beta}}  \leq |\Gamma_2^M (u) (0) - \Gamma_2 (u) (0) | < \frac{\rho}{4}.
\end{equation*}
\underline{II: $1 \leq l \leq j \leq N, j+l \geq N+1$:} We will make use of the basic estimate 
\begin{align*}
\left(v_N(j\pm l)-v_N(j)-v_N(l)\right)^2 \lesssim v_N(l)^2 + \left( v_N(j\pm l) - v_N(j)\right)^2 .
\end{align*}
Observe that 
\begin{align*}
\sum\limits_{j=\frac{N}{2}}^\infty \frac{1}{j^{1+\beta}} \sum\limits_{l=1}^j \frac{v_N(l)^2}{l^{1+\beta}} \leq \sum\limits_{j=\frac{N}{2}}^\infty \frac{1}{j^{1+\beta}} \sum\limits_{l=1}^j \frac{l^{2\beta - 2\varepsilon}}{l^{1+\beta}} \lesssim  \sum\limits_{j= \frac{N}{2}}^\infty \frac{1}{j^{1+2\varepsilon}} \xlongrightarrow{N \to \infty} 0 ,
\end{align*}
where we applied \eqref{eq:intvergleichlarge} in the second step with $\gamma= \beta - 2\varepsilon -1$. Note that we will also apply this estimate to the term involving `$j-l$' later.
If $\beta>1$ we have by \eqref{eq:mixedgeq1}
\begin{align*}
\left( v_N(j+l) - v_N(j) \right)^2 \lesssim N^{2\beta - 2\varepsilon -2} l^2 \leq  \left(j+l\right)^{2\beta - 2\varepsilon -2}l^2 \leq \left( 2 j \right)^{2\beta - 2\varepsilon -2} l^2 .
\end{align*}
Due to \eqref{eq:mixedleq1}, we can thus conclude for any $\beta \in (0,2)$ it holds
\begin{align*}
\left( v_N(j+l) - v_N(j) \right)^2 \lesssim j^{2\beta-2\varepsilon-2} l^2  .
\end{align*}
Hence, we obtain with the help of \eqref{eq:intvergleichlarge} for $\gamma=1-\beta$ the upper bound
\begin{align*}
\sum\limits_{j= \frac{N}{2}+1}^N \frac{1}{j^{1+\beta}} \sum\limits_{l=N-j+1}^j \frac{\left( v_N(j+l)-v_N(j)\right)^2}{l^{1+\beta}} &\lesssim \sum\limits_{j=\frac{N}{2}+1}^N \frac{j^{2\beta - 2\varepsilon -2}}{j^{1+\beta}} \sum\limits_{l=1}^j l^{1-\beta} \lesssim \sum\limits_{j= \frac{N}{2}+1}^N \frac{1}{j^{1+2\varepsilon}}
\end{align*}
which tends to zero as $N \to \infty$.\\
\underline{III: $1\leq l \leq j, j \geq N+1$:} Since $N < j \leq j+l$ we observe by \eqref{eq:nonmixedmvt}
\begin{align*}
\sum\limits_{j=N+1}^\infty \frac{1}{j^{1+\beta}} \sum\limits_{l=1}^j \frac{\left(v_N(j+l)-v_N(j)\right)^2}{l^{1+\beta}} 
&\lesssim N^{2\beta - 2\varepsilon -4} \!\sum\limits_{j=N+1}^{N^2-1} \frac{1}{j^{1+\beta}} \sum\limits_{l=1}^j l^{1-\beta}
\lesssim N^{2\beta - 2\varepsilon -4}\! \sum\limits_{j=N+1}^{N^2-1} j^{1-2\beta},
\end{align*}
where we applied \eqref{eq:intvergleichlarge} in the last step. According to \eqref{eq:intvergleichlarge} and \eqref{eq:intvergleichlow} we find
\begin{align*}
\sum\limits_{j=N+1}^{N^2-1} j^{1-2\beta} \lesssim \begin{cases} N^{4- 4\beta} , & \beta < 1 \\
\log(N^2), & \beta = 1\\
N^{2-2\beta}, & \beta > 1.
\end{cases}
\end{align*}
Multiplying each of these expressions by $N^{2\beta- 2\varepsilon -4}$ yields the desired convergence  for any $\beta \in (0,2)$. Hence, the first part of \eqref{eq:Gammatriangle} converges to zero as $N$ tends to infinity.

For the estimates of the second part, we need a refined splitting of the domain, sketched in Figure \ref{fig:domain2}.
\begin{figure}[ht]
 \centering
\begin{tikzpicture}[scale = 1]
\footnotesize
 \draw (0,0) -- (7,7);
\draw[->,thick] (0,-0.3) -- (0,7);
 \draw[->, thick] (-0.3,0) -- (8,0);

 \draw (1,0) -- (1,1);
 \draw (3,0)-- (3,3) -- (8,3);
\draw (3,0) -- (6,3);

\node at (2,1) {A};
\node (B) at (5,6) {B};
\node at (4.0,1.8) {C};
\node at (7,1.5) {D};
\node at (7,4.5) {E};

\draw[->] (B) to [bend right](5,5);

 \node  at (-0.4, 5) {$l$};
 \node  at (7, -0.4) {$j$};
 \node  at (1,-0.4) {M};
 \node  at (3,-0.4) {N};
   \node  at (-0.5,3) {N};
   \node at (0,3) {-};
      \node  at (-0.5,1) {M};
   \node at (0,1) {-};
\end{tikzpicture}
 \caption{Splitting of the domain of summation for the term with $j-l$.}\label{fig:domain2}
\end{figure}
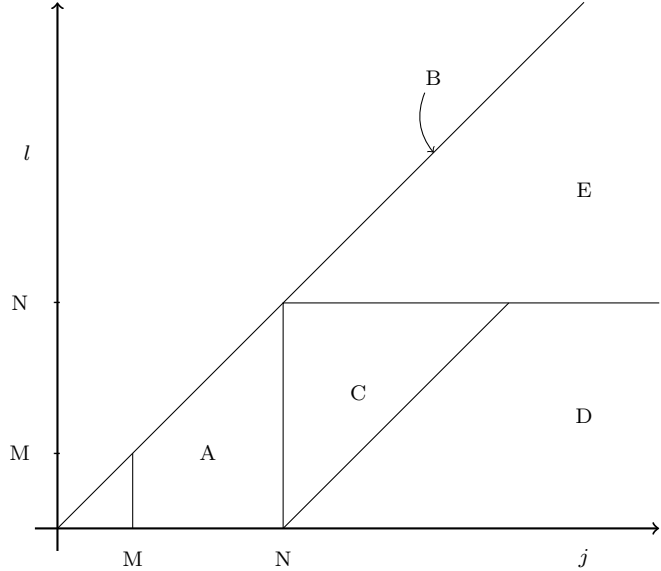\\
\underline{A: $1\leq l \leq j \leq N, j \geq M+1$:} Since $j-l \leq N$ it follows similar to the case I that
\begin{equation*}
\sum\limits_{j=M+1}^N \sum\limits_{l=1}^j \frac{\left( v_N(j-l) - v_N(j)-v_N(l)\right)^2}{j^{1+\beta}l^{1+\beta}}  \leq |\Gamma_2^M (u) (0) - \Gamma_2 (u) (0) |  < \frac{\rho}{4}.
\end{equation*}
\underline{B: $l= j \geq N+1$:} 
In this situation  we have  due to $v_N(j)^2 \leq u(j)^2$ the estimate
\begin{align*}
\sum\limits_{j=N+1}^\infty \frac{\left(v_N(0)- 2v_N(j)\right)^2}{j^{2+2\beta}} = 4 \sum\limits_{j=N+1}^\infty \frac{1}{j^{2+2\beta}} v_N(j)^2 \leq 4 \sum\limits_{j=N+1}^\infty \frac{1}{j^{2+2\varepsilon}} \xlongrightarrow{N \to \infty} 0.
\end{align*}
\underline{C, D, E: $1 \leq l \leq j-1,  j \geq N+1$, where  $\beta>1$:} We observe in this case by \eqref{eq:mixedgeq1} and \eqref{eq:intvergleichlarge}
\begin{align*}
&\sum\limits_{j=N+1}^\infty \frac{1}{j^{1+\beta}} \sum\limits_{l=1}^j \frac{\left( v_N(j) - v_N(j-l)\right)^2}{l^{1+\beta}} \lesssim N^{2\beta - 2\varepsilon -2} \sum\limits_{j=N+1}^\infty \frac{1}{j^{1+\beta}} \sum_{l=1}^j l^{1-\beta} \\
 & \qquad \lesssim N^{2\beta -2\varepsilon -2} \sum\limits_{j=N+1}^\infty \frac{1}{j^{2\beta -1}} \lesssim \frac{1}{N^{2\varepsilon}}  \to 0 \quad  (N \to \infty),
\end{align*}
where we applied $\beta > 1$ and \eqref{eq:intvergleichlow} in the last step. Recalling the estimates from case II, we have established the claim for $\beta>1$.
We assume from now on $\beta \leq 1$.\\
\underline{C: $1 \leq l \leq N,  j \geq N+1,  j-l \leq N$, where $\beta\leq1$:} Due to \eqref{eq:mixedleq1} we observe
\begin{align*}
\sum\limits_{j=N+1}^{2N} \frac{1}{j^{1+\beta}} \sum\limits_{l=j-N}^{N} \frac{\left( v_N(j) - v_N(j-l)\right)^2}{l^{1+\beta}} &\lesssim  \sum\limits_{j=N+1}^{2N}\frac{1}{j^{1+\beta}} \sum\limits_{l=j-N}^N l^{1-\beta} \left(j-l\right)^{2\beta - 2\varepsilon -2}\\
 &\lesssim N^{1-\beta} \sum\limits_{j=N+1}^{2N} \frac{1}{j^{1+\beta}} \sum\limits_{m=j-N}^{N} m^{2\beta - 2\varepsilon -2}.
\end{align*}
In order to treat the latter sum, we proceed with a finer {case} separation. Note that we can exclude the case of $2\beta-2\varepsilon -1 = 0$ by choosing $\varepsilon>0$ appropriately small. If $2 \beta -2\varepsilon -1 > 0$ we can apply \eqref{eq:intvergleichlarge} and prove its convergence to zero similar as before. 
To conclude convergence in the case of $2 \beta -2\varepsilon -1 < 0$  we  {split} the expression and use \eqref{eq:intvergleichlarge} and \eqref{eq:intvergleichlow} as follows:
\begin{align*}
N^{1-\beta} \sum\limits_{j=N+1}^{2N} &\frac{1}{j^{1+\beta}} \sum\limits_{m=j-N}^{N} m^{2\beta - 2\varepsilon -2}\\ 
&= N^{1-\beta} \left(\sum\limits_{j=N+2}^{2N} \frac{1}{j^{1+\beta}} \sum\limits_{m=j-N}^{N} m^{2\beta - 2\varepsilon -2} + \frac{1}{(N+1)^{1+\beta}} \left(1+\sum\limits_{m=2}^{N} m^{2\beta - 2\varepsilon -2} \right) \right)\\
&\lesssim N^{1-\beta}\!\! \sum\limits_{j=N+2}^{2N} \frac{\left(j-N-1\right)^{2\beta -2\varepsilon -1}}{j^{1+\beta}} + \frac{N^{1-\beta}}{N^{1+ \beta}} \\
&\lesssim \frac{1}{N^{2\beta}} \sum\limits_{m=1}^{N-1}m^{2\beta - 2\varepsilon -1} + \frac{2}{N^{2\beta}}
\lesssim \frac{1}{N^{2\varepsilon}} + \frac{1}{N^{2\beta}} \xlongrightarrow{N \to \infty} 0.
\end{align*}
\underline{D: $1 \leq l \leq N < j$, $j-l \geq N+1$, where $\beta\leq1$:} Here we have the comfortable situation that $j , j-l > N$. Therefore, we can apply \eqref{eq:nonmixedmvt} and observe
\begin{align*}
\sum\limits_{l=1}^N \frac{1}{l^{1+\beta}} \sum\limits_{j=N+l+1}^\infty \frac{\left(v_N(j)-v_N(j-l)\right)^2}{j^{1+\beta}} 
&\lesssim N^{2\beta - 2\varepsilon -4} \sum\limits_{l=1}^N l^{1-\beta} \sum\limits_{j=N+l+1}^\infty \frac{1}{j^{1+\beta}} \\
&\lesssim N^{\beta - 2\varepsilon -4} \sum\limits_{l=1}^N l^{1-\beta}\lesssim \frac{1}{N^{2+2\varepsilon}} \xlongrightarrow{N \to \infty} 0,
\end{align*}
where we applied \eqref{eq:intvergleichlow} to the inner sum and \eqref{eq:intvergleichlarge} to the outer sum. Due to the arguments of the case II, we establish the claim in this situation.\\
\underline{E: $N+1 \leq l \leq j-1,  j \geq N+1$, where $\beta\leq1$:} Since $\max\limits_{x \in \mathbb{Z}} v_N(x) = N^{\beta -\varepsilon}$ we can estimate
\begin{align*}
\sum\limits_{j=N+1}^\infty \frac{1}{j^{1+\beta}} \sum\limits_{l=N+1}^{j-1} \frac{\left(v_N(j-l) - v_N(j) -v_N(l) \right)^2}{l^{1+\beta}} &\lesssim N^{2\beta - 2\varepsilon} \sum\limits_{j=N+1}^\infty \frac{1}{j^{1+\beta}} \sum\limits_{l=N+1}^{j-1} \frac{1}{l^{1+\beta}} \lesssim \frac{1}{N^{2\varepsilon}}, 
\end{align*}
where we applied \eqref{eq:intvergleichlow} to both sums.
Whence, we find $| \Gamma_2 (v_N)(0) - \Gamma_2^M (v_N)(0)| < \frac{3 \rho}{4}$ for $N \geq  N_0(\rho) $ and thus conclude the claim.
\end{proof}

\section{Kernels with finite second moment}
We have seen in the previous section that for power type kernels with $\beta<2$ the inequality $CD(0,d)$ fails for all
finite $d>0$. It turns out that a positive result can be obtained if $\beta>2$. This is a consequence of the following theorem,
where an even more general class of kernels is admissible.  Recall that we always assume that the kernel $k$ is subject to the conditions (K1) and (K2) with $|k|_1>0$. 
\begin{satz} \label{posresult}
Let $k$ be a kernel with finite second moment, that is
\[
\sum_{j\in \iN} k(j) j^2<\infty,
\]
and assume that $k$ is non-increasing on $\iN$. Then the corresponding operator $\cL$ satisfies $CD(0,d)$ for some finite $d > 0$ of the form \[d=C_0 \cdot \displaystyle \frac{|k|_1 \sum_j k(j)j^2}{k(1)^2},\]
where $C_0>0$ is some constant independent of $k$.
\end{satz}
\begin{proof} 
The proof consists of several steps.
By Lemma \ref{lem:symmetry}, we may assume that $u(0)=0$ as well as $u(-j)=u(j)$ for all $j\in \iZ$.

{\em Step 1: H\"older estimate}. Using that $k$ has finite second moment, we have by H\"older's inequality that
\begin{align}
\big(\cL(u)(0)\big)^2 & =4 \big(\sum_{j=1}^\infty k(j)u(j)\big)^2= 4 \big(\sum_{j=1}^\infty 
\sqrt{k(j)}j\, \frac{\sqrt{k(j)} u(j)}{j} \big)^2\nonumber\\
& \le 4 \sum_{j=1}^\infty 
k(j) j^2 \sum_{j=1}^\infty \frac{k(j)}{j^2}\,u(j)^2
= C(k) \sum_{j=1}^\infty  \frac{k(j)}{j^2}\,u(j)^2. \label{beta>2Hoelder} 
\end{align}

{\em Step 2: Basic lower estimate for $\Gamma_2$}. Evidently,
\[
2 \Gamma_2(u)(0)\ge \sum_{j=1}^\infty k(1)k(j) \big(u(j+1)-u(j)-u(1)\big)^2,
\]
and therefore
\begin{align}
\sum_{j=1}^\infty k(j) \big(u(j+1)-u(j)\big)^2 &  \le 2\sum_{j=1}^\infty k(j) \big(u(j+1)-u(j)-u(1)\big)^2 
+2\sum_{j=1}^\infty k(j) u(1)^2\nonumber\\
& \le \,\frac{4}{k(1)}\,\Gamma_2(u)(0)+|k|_1 u(1)^2,\label{beta>2step2a}
\end{align}
where we use the fact that 
\[
a^2=(a-b+b)^2\le 2(a-b)^2+2b^2,\quad a,b\in \iR.
\]
From the basic estimate \eqref{basicGamma2Est} we know that $2k(1)^2 u(1)^2\le \Gamma_2(u)(0)$, which together
with \eqref{beta>2step2a} implies
\begin{equation} \label{beta>2step2b}
\sum_{j=1}^\infty k(j) \big(u(j+1)-u(j)\big)^2 \le \big(\frac{4}{k(1)}+\frac{|k|_1}{2k(1)^2}   \big) \Gamma_2(u)(0).
\end{equation}
Observe that the assumptions on the kernel ensure that $k(1)>0$.

{\em Step 3: Estimating a squared weighted $l_2$-norm of  $u$ by $\Gamma_2$}. Let $N\in \iN$ with $N\ge 2$. For any $\delta>0$ we have by Young's inequality and 
\eqref{beta>2step2b} that
\begin{align*}
\sum_{j=1}^N & \frac{k(j)}{j}\,u(j+1)^2=
\sum_{j=1}^N \frac{k(j)}{j}\,\big(u(j+1)-u(j)+u(j)\big)^2\\
& \le \sum_{j=1}^N \frac{k(j)}{j}\,\Big[\big(1+\frac{j}{\delta}\big)\big(u(j+1)-u(j)\big)^2+\big(1+\frac{\delta}{j}\big) u(j)^2\Big]\\
& \le \sum_{j=1}^N  k(j)\,\big(1+\frac{1}{\delta}\big)\big(u(j+1)-u(j)\big)^2+k(1)\,(1+\delta)
u(1)^2\\
& \quad +\sum_{j=1}^{N-1} \frac{k(j+1)}{j+1}\,\big(1+\frac{\delta}{j+1}\big) u(j+1)^2\\
& \le M_\delta \Gamma_2(u)(0)
+\sum_{j=1}^{N} \frac{k(j+1)}{j+1}\,\big(1+\frac{\delta}{j+1}\big) u(j+1)^2,
\end{align*}
with
\[
M_\delta=\big(1+\frac{1}{\delta}\big)\big(\frac{4}{k(1)}+\frac{|k|_1}{2k(1)^2}\big)
+\frac{1}{2k(1)}\,(1+\delta).
\]
This implies that
\begin{equation} \label{beta>2step3a}
\sum_{j=1}^N \Lambda(j,\delta)
\,u(j+1)^2\le M_\delta \Gamma_2(u)(0),
\end{equation}
where
\begin{align*}
\Lambda(j,\delta) & =\frac{k(j)}{j}-\frac{k(j+1)}{j+1}-\delta\,\frac{ k(j+1)}{(j+1)^2}\\
& \ge k(j+1)\Big(\frac{1}{j}-\frac{1}{j+1}-\frac{\delta}{(j+1)^2} \Big)\\
& \ge  k(j+1)\,\frac{1-\delta}{(j+1)^2},
\end{align*}
by monotonicity of the kernel $k$.
Choosing $\delta\in (0,1)$, it follows from \eqref{beta>2step3a} that
\[
\sum_{j=1}^N \frac{k(j+1)}{(j+1)^2}
\,u(j+1)^2\le \frac{M_\delta}{1-\delta}\, \Gamma_2(u)(0).
\]
Sending $N\to \infty$ we obtain
\begin{equation} \label{beta>2step3b}
\sum_{j=1}^\infty  \frac{k(j+1)}{(j+1)^2}
\,u(j+1)^2\le \frac{M_\delta}{1-\delta}\, \Gamma_2(u)(0).
\end{equation}

{\em Step 4: Combining the estimates from Step 1 and 3}. Using \eqref{beta>2Hoelder}, \eqref{beta>2step3b}
and $2k(1)^2 u(1)^2\le \Gamma_2(u)(0)$
we obtain
\begin{align*}
\big(\cL(u)(0)\big)^2 & \le  C(k) \Big( k(1)\,u(1)^2 + \sum_{j=1}^\infty \frac{k(j+1)}{(j+1)^2}\,u(j+1)^2\Big)\\
& \le C(k) \Big(\frac{1}{2 k(1)}+\frac{M_\delta}{1-\delta} \Big)  \Gamma_2(u)(0).
\end{align*}
The proof of Theorem \ref{posresult} is complete.
\end{proof}

Theorem \ref{posresult} covers a wide class of non-increasing (on $\iN$) kernels which decay sufficiently fast. We illustrate this in the following corollary. 
\begin{korollar} \label{posexamples}
If the kernel $k$ belongs to one of the subsequent classes of kernels, the corresponding 
operator $\cL$ satisfies the $CD(0,d)$-inequality for some finite $d>0$.
\begin{itemize}
\item[(i)] {\em Power type kernels}: 
\[
k(j)=\frac{c}{|j|^{1+\beta}},\quad j\in \iZ\setminus \{0\},
\]
where $c>0$ and $\beta>2$.
\item[(ii)] {\em Exponential kernels}:
\[
k(j)=c \,e^{- \delta |j|^\alpha},\quad j\in \iZ\setminus \{0\},
\]
where $c,\alpha,\delta>0$. 
\item[(iii)] {\em Mixed exponential and power type:}
\[
k(j)=c \,\frac{e^{- \delta |j|}}{|j|^\alpha},\quad j\in \iZ\setminus \{0\},
\]
where $c,\alpha,\delta>0$. 
\end{itemize}
\end{korollar}
\begin{proof}
It is easily seen that all the described kernels satisfy the assumptions of Theorem \ref{posresult}.
\end{proof}
\begin{bemerk1}
{\em
Estimate \eqref{beta>2step3b} is one of the key estimates in the proof of Theorem \ref{posresult}. 
In the case of the power type kernel with $\beta>2$
it can be rephrased as
\[
\sum_{j=2}^\infty k(j)^{\frac{3+\beta}{1+\beta}}u(j)^2 \le M \Gamma_2(u)(0)
\]
(for symmetric $u$) with some constant $M>0$. This is a much stronger estimate than the basic inequality 
$\sum_{j=2}^\infty k(j)^2 u(j)^2 \le C \Gamma_2(u)(0)$ from \eqref{basicGamma2Est}, which holds for any kernel and symmetric $u$.
The difference becomes more apparent for large values of $\beta$, since the exponent $\frac{3+\beta}{1+\beta}$ tends to
$1$ as $\beta\to \infty$. 
}
\end{bemerk1}
\begin{bemerk1} \label{extensionposresult}
{\em The statement of Theorem \ref{posresult} (with adapted bound for $d$) remains true in the more general case where $k$ has a finite second moment 
but is merely assumed to be non-increasing for all $j\ge j_0$ with some fixed $j_0\in \iN$. 
The key idea to see this in the case $j_0>1$ is to split the sum
\begin{equation} \label{ext1}
\sum_{j=1}^{\infty} \frac{k(j+1)}{(j+1)^2}\,u(j+1)^2=\sum_{j=1}^{j_0-1} \frac{k(j+1)}{(j+1)^2}\,u(j+1)^2
+\sum_{j=j_0}^{\infty} \frac{k(j+1)}{(j+1)^2}\,u(j+1)^2.
\end{equation}
The second sum on the right-hand side can be estimated from above by $\Gamma_2(u)(0)$ analogously to the argument given in the proof of Theorem \ref{posresult}. Instead of the $u(1)^2$ term one has a term involving $u(j_0)^2$, which can be estimated by employing \eqref{basicGamma2Est}. The finite sum on the right of \eqref{ext1} can also be controlled from above by $\Gamma_2(u)(0)$, again by the basic estimate  \eqref{basicGamma2Est}. 
}
\end{bemerk1}
\begin{bei}
{\em 
Let $k$ be of the form
\[
k(j)=c\,|j|^\alpha e^{-\delta |j|},\quad j\in \iZ,
\]
with $c,\alpha,\delta>0$. Then the associated 
operator $\cL$ satisfies the $CD(0,d)$-inequality for some finite $d>0$. This is a consequence of Theorem \ref{posresult}
and Remark \ref{extensionposresult}.
}
\end{bei}

Theorem \ref{betalessthan2} and Theorem \ref{posresult} show that for power type kernels the value $\beta=2$ is a critical
case. While $\beta>2$ ensures $CD(0,d)$ for some finite $d>0$, the letter fails to be true in the case $\beta<2$.
So what happens in the case $\beta=2$? This is an interesting question, which we are not able to answer at the moment.
However, if we add to a power type kernel with $\beta=2$ a suitable logarithmic factor, then $CD(0,d)$ is still valid for some finite $d>0$ as the resulting kernel has a finite second moment and thus Theorem \ref{posresult} applies. We formulate
this observation in the following corollary.  
\begin{korollar} \label{critcase}
Suppose that the kernel $k$ is of the form
\[
k(j)=\,\frac{c}{|j|^{3}[1+\log(|j|)]^{1+\alpha}},\quad j\in \iZ\setminus \{0\},
\]
with $c,\alpha>0$. Then the corresponding 
operator $\cL$ satisfies the $CD(0,d)$-inequality for some finite $d>0$.
\end{korollar}
\section{Kernels with sparse unbounded support}
In \label{Sect:Sparse}this section we consider situations where the support of the kernel $k$ is unbounded but relatively small, in particular
$k$ is no longer non-increasing on $\iN$. We begin with two positive results. Throughout this section the support of $k$ 
is denoted by $S$ and $S_+=S\cap \iN$.
\begin{satz} \label{powersof2}
Assume that the support of the kernel $k$ is given by $S=\{\pm 2^l:\,l\in \iN_0\}$ and that $k$ satisfies the condition
\begin{equation} \label{S2}
\sum_{l\in S_+} \Big(\frac{k(2l)}{k(l)}\Big)^2\,<\infty.
\end{equation}
Then the corresponding operator $\cL$ satisfies $CD(0,d)$ for some finite $d>0$.
\end{satz}
\begin{proof}
By Lemma \ref{lem:symmetry}, we may assume that $u(-j)=u(j)$ for all $j\in \iN$ as well as $u(0)=0$. In this situation
\begin{align*}
\big(\cL(u)(0)\big)^2 & = \big(2\sum_{j\in S_+}k(2j)u(2j)+2k(1)u(1) \big)^2\\
& \le 8 \big(\sum_{j\in S_+}k(2j)u(2j)\big)^2+8k(1)^2 u(1)^2\\
& \le 8 \big(\sum_{j\in S_+}k(2j)u(2j)\big)^2+4\Gamma_2(u)(0),
\end{align*}
by the basic inequality \eqref{basicGamma2Est}. The first term on the right-hand side can be estimated by H\"older's
inequality using \eqref{S2} as follows.
\begin{align*}
\big(\sum_{j\in S_+}k(2j)u(2j)\big)^2 & =\big(\sum_{j\in S_+}\frac{k(2j)}{k(j)}\, k(j) u(2j)\big)^2 \le \sum_{j\in S_+}\Big(\frac{k(2j)}{k(j)}\Big)^2\, \sum_{j\in S_+}k(j)^2 u(2j)^2.
\end{align*}
Consequently, 
\begin{align} \label{powers2a}
\big(\cL(u)(0)\big)^2 & \le C\big( \sum_{j\in S_+}k(j)^2 u(2j)^2+\Gamma_2(u)(0)\big),
\end{align}
with some constant $C$ which only depends on the kernel.

Next, using the inequality $a^2\le 2(a-b)^2+2b^2$, $a,b\in \iR$, we have
\begin{align*}
\sum_{j\in S_+}k(j)^2 u(2j)^2 & \le 2\sum_{j\in S_+}k(j)^2 \big(u(2j)-2u(j)\big)^2 
+ 8\sum_{j\in S_+}k(j)^2 u(j)^2\\
& = 2\sum_{j\in S_+}k(j)^2 \big(u(j+j)-u(j)-u(j)\big)^2 
+ 8\sum_{j\in S_+}k(j)^2 u(j)^2\\
& \le 4\Gamma_2(u)(0)+4\Gamma_2(u)(0)=8 \Gamma_2(u)(0),
\end{align*}
by the representation formula for $\Gamma_2$ and the basic inequality \eqref{basicGamma2Est}.
Combining the last estimate and \eqref{powers2a} yields the asserted CD-inequality.
\end{proof}
\begin{bei} \label{examplepowers2}
{\em
All kernels $k$ of exponential type or mixed exponential and power type (see Corollary \ref{posexamples}, (ii) and (iii)) satisfy 
condition \eqref{S2}, whereas purely algebraic kernels (that is, kernels of power type as described in
\eqref{powerkerneldef})
do not possess this property.
}
\end{bei}

It turns out that a $CD(0,d)$-inequality with finite $d>0$ is still possible for larger gaps between the elements of the support.
The extreme case is described in the following result.
\begin{satz} \label{powersof3}
Assume that the support of $k$ is given by $S=\{\pm 3^l:\,l\in \iN_0\}$ and that $k$ is subject to the condition
\begin{equation} \label{S3}
\sum_{l\in S_+} \frac{k(3l)}{k(l)}\,<\infty.
\end{equation}
Then the corresponding operator $\cL$ satisfies $CD(0,d)$ for some finite $d>0$.
\end{satz}
\begin{proof}
Assuming w.l.o.g.\ that $u(0)=0$ and $u(-j)=u(j)$ for all $j\in \iN$ we have analogously to the first part of the proof of Theorem
\ref{powersof2} that
\begin{align*}
\big(\cL(u)(0)\big)^2  
& \le 8 \big(\sum_{j\in S_+}k(3j)u(3j)\big)^2+4\Gamma_2(u)(0).
\end{align*}
Using \eqref{S3}, H\"older's
inequality gives
\begin{align*}
\big(\sum_{j\in S_+}k(3j)u(3j)\big)^2 & =\big(\sum_{j\in S_+}\sqrt{\frac{k(3j)}{k(j)}}\, \sqrt{k(3j) k(j)} u(3j)\big)^2 \\
& \le \sum_{j\in S_+}\,\frac{k(3j)}{k(j)}\, \sum_{j\in S_+}k(3j)k(j) u(3j)^2,
\end{align*}
and thus there is a constant $C>0$ depending only on $k$ such that
\begin{equation} \label{S3a}
\big(\cL(u)(0)\big)^2  \le C\big( \sum_{j\in S_+}k(3j)k(j) u(3j)^2+\Gamma_2(u)(0)\big).
\end{equation}
The key idea of the proof is the following decomposition and estimate of the term $u(3j)^2$,
where we use the fact that $(a+b+c)^2\le 3(a^2+b^2+c^2)$ for all $a,b,c\in \iR$.
\begin{align*}
u(3j)^2 & = \Big(\big(u(3j)-u(2j)+u(j)\big)+\big(u(2j)-2u(j)\big)+u(j)\Big)^2\\
& \le 3\big(u(3j)-u(2j)+u(j)\big)^2+3\big(u(2j)-2u(j)\big)^2+3 u(j)^2.
\end{align*}
From \eqref{S3} we obtain some constant $c_0 > 0$ such that $k(3j)\le c_0 k(j)$ for all $j\in S_+$. Using this comparability property and the symmetry of $u$ and $k$ this allows us to estimate
\begin{align*}
\sum_{j\in S_+}k(3j)k(j) u(3j)^2 & \le 3\sum_{j\in S_+}k(3j)k(-j)\big(u(3j-j)-u(3j)-u(-j)\big)^2 \\
& \quad\; +3c_0\sum_{j\in S_+}k(j)^2\big(u(2j)-2u(j)\big)^2 +3c_0\sum_{j\in S_+}k(j)^2 u(j)^2\\
& \le \big(6+6c_0+\frac{3}{2}\,c_0\big)\,\Gamma_2(u)(0).
\end{align*}
The last estimate and \eqref{S3a} yield the assertion.
\end{proof}
\begin{bei} \label{examplepowers3}
{\em
All kernels $k$ of exponential type or mixed exponential and power type (see Corollary \ref{posexamples}, (ii) and (iii)) satisfy 
condition \eqref{S3}, whereas purely algebraic kernels (that is, kernels of power type as described in
\eqref{powerkerneldef})
do not enjoy this property.
}
\end{bei}

Writing $S_+=\{x_l:\,l\in \iN\}$ where the sequence $x_l$ is strictly increasing, we have $x_{l+1}=3x_l$ for all $l\in \iN$
in the situation of Theorem \ref{powersof3}. This is an extreme case as the next result will imply that for {\em any} kernel
$k$ with 
\[
x_{l+1}\ge 3 x_l+1,\quad l\in \iN,
\]
the $CD(0,d)$-inequality fails for all finite $d>0$. 
\begin{satz} \label{thinsupport}
Let $k$ be a kernel with unbounded support $S$ such that the following conditions are satisfied.
\begin{itemize}
\item[(i)] There exists a number $N\in \iN_0$ such that for all $j,l\in S$ with $\max\{|j|,|l|\}>N$ we have $j+l\notin S$.
\item[(ii)] For every number $m\in \iN$ with $m>2N$ there exists at most one decomposition $m=j+l$ with
$j,l\in S$. 
\item[(iii)] For every number $m\in \iN$ with $m\le 2N$ there exist at most finitely many decompositions of the form $m=j+l$ with
$j,l\in S$. 
\end{itemize}
Then the operator $\cL$ associated with $k$ fails to satisfy $CD(0,d)$ for all finite $d>0$.
\end{satz}
\begin{proof}
Letting $n> 2N$ we define the function $u_n:\iZ\to \iR$ as follows. We set $u_n(j)=0$ for every $j\in \iN_0$ with $j\le 2N$. 
If $j>2N$ we distinguish several cases. Suppose first that $j\in S_+$. Then we set $u_n(j)=\frac{1}{k(j)}$ whenever
$j\le n$ and otherwise $u_n(j)=0$. Now suppose that $j\notin S_+$ and $j=l_1+l_2$
with (unique) $l_1,l_2\in S$. In this case we set $u_n(j)=u_n(l_1)+u_n(l_2)$. For all remaining $j>2N$ we set $u_n(j)=0$.
Finally, we put $u_n(-j)=u_n(j)$ for all $j\in\iN$. Observe that $u_n$ is well defined, in view of the assumptions (i) and (ii). 

We have now
\[
\cL(u_n)(0)=2\sum_{j=2N+1}^{n}k(j)u_n(j)=2\sum_{j=2N+1}^{n} \chi_{S_+}(j),
\]
where $\chi_{S_+}$ denotes the characteristic function of the set $S_+$.

On the other hand,
\begin{align*}
4\Gamma_2(u_n)(0) & =\sum_{j,l\in S}k(j)k(l)\big(u_n(j+l)-u_n(j)-u_n(l)\big)^2\\
& = 4 \sum_{j\in S}k(j)^2 u_n(j)^2
 +\sum_{j,l\in S, l\neq -j}k(j)k(l)\big(u_n(j+l)-u_n(j)-u_n(l)\big)^2\\
 & = 8 \sum_{j=2N+1}^{n} \chi_{S_+}(j)+\sum_{j,l\in S, 1\le |j+l|\le 2N}k(j)k(l)\big(u_n(j+l)-u_n(j)-u_n(l)\big)^2\\
 & \quad +\sum_{j,l\in S, |j+l|> 2N}k(j)k(l)\big(u_n(j+l)-u_n(j)-u_n(l)\big)^2.
\end{align*}
The last term is equal to zero, by construction of $u_n$, and the second to the last one is bounded from above by a constant $M$
which is independent of $n$, by assumption (iii). 
Since the support of $k$ is unbounded, we have $\xi_n:=\sum_{j=2N+1}^{n} \chi_{S_+}(j)\to \infty$ as $n\to \infty$ 
and thus it follows that
\[
\limsup_{n\to \infty}\frac{4\Gamma_2(u_n)(0)}{\big(\cL(u_n)(0)\big)^2}\le 
\limsup_{n\to \infty}\frac{8\xi_n+M}{4\xi_n^2} =0,
\]
which proves the theorem.
\end{proof}

\begin{bemerk1}
\label{rem:infiniteproduct}
In the situation of Theorem \ref{thinsupport} with $N=0$, the underlying graph $\mathcal G$ (described in the introduction after \eqref{Gamma2Formel}) has a simple product structure. More precisely, the first and second neighbourhood of any vertex $x \in \mathbb{Z}$ are isomorphic to the first and second neighbourhood, respectively, to $(0,0,\ldots)$ in the infinite Cartesian product $\bigtimes_{j \in S_+} k(j) \mathbb{Z}$. Note that, by definition, $CD(\kappa,d)$ at $x$ only depends on the first and second neighbourhood of $x$.
\end{bemerk1}

\section{Is a positive curvature possible?}
In this section we show that $CD(\kappa,\infty)$ with $\kappa>0$ does not hold for arbitrary kernels $k$ provided that
$k$ is non-increasing on $\iN$.
\begin{satz}\label{theo:nocurvature}
Suppose that the kernel $k$ is non-increasing on $\iN$. Then the corresponding operator $\cL$ fails to satisfy
the $CD(\kappa,\infty)$-condition for all $\kappa>0$. This is even true when the class of admissible functions
is restricted to compactly supported functions.
\end{satz}

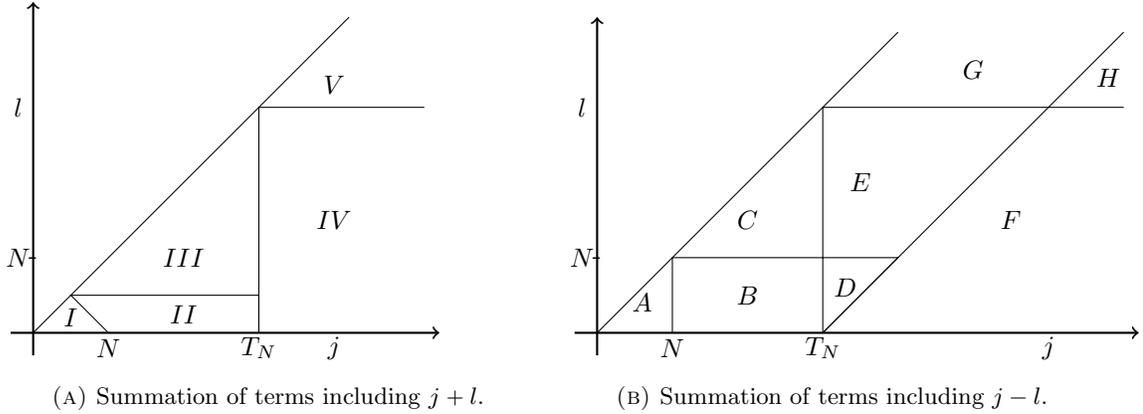
\begin{figure}[ht]
    \centering
    \begin{subfigure}[b]{0.49\textwidth}
     \begin{tikzpicture}[scale = 1]
 \draw (0,0) -- (4.2,4.2);
  \draw[->,thick] (0,-0.3) -- (0,4.4);
 \draw[->, thick] (-0.3,0) -- (5.4,0);
  \node (j) at (-0.2, 3) {$l$};
 \node (l) at (4, -0.2) {$j$};
 \node (n3) at (1,-0.2) {$N$};
 \node (n4) at (3,-0.2) {$T_N$};
 \node at (-0.2,1) {$N$};
   \node  at (0,1) {-};

\draw (1,0) -- (0.5,0.5) -- (3,0.5)--(3,0)--(3,3)--(5.2,3);

\node at (0.5,0.2) {$I$};
\node at (2,0.25) {$II$};
\node at (2,1) {$III$};
\node at (4,1.5) {$IV$};
\node at (4,3.3) {$V$};

   \end{tikzpicture}
 \caption{Summation of terms including $j+l$.}
  \label{fig:domain2a}
    \end{subfigure}
    \hfill
    \begin{subfigure}[b]{0.49\textwidth}
    \begin{tikzpicture}[scale = 1]
 \draw (0,0) -- (4,4);
  \draw[->,thick] (0,-0.3) -- (0,4.2);
 \draw[->, thick] (-0.3,0) -- (7.2,0);
  \node (j) at (-0.2, 3) {$l$};
 \node (l) at (6, -0.2) {$j$};
 \node (n3) at (1,-0.2) {$N$};
 \node (n4) at (3,-0.2) {$T_N$};
 \node at (-0.2,1) {$N$};
   \node  at (0,1) {-};

\draw (1,0) -- (1,1) -- (4,1)--(3,0)--(3,3)--(7,3);
\draw (3,0) -- (7,4);

\node at (0.6,0.4) {$A$};
\node at (2,0.5) {$B$};
\node at (2,1.5) {$C$};
\node at (3.3,0.6) {$D$};
\node at (3.5,2) {$E$};
\node at (5.5,1.5) {$F$};
\node at (5,3.5) {$G$};
\node at (6.8,3.4) {$H$};

   \end{tikzpicture}
 \caption{Summation of terms including $j-l$.}
  \label{fig:domain1a}
    \end{subfigure}
\caption{Splitting the domain of summation.}
\label{fig:3}
\end{figure}

\begin{proof}
Note that for kernels with finite second moment we immediately obtain the claim for the admissible, but non-compactly supported function $u(j) = j$. To show the desired result for all non-increasing kernels we use the following cut-off procedure.
Let $\left(T_N\right)_{N \in \mathbb{N}} $ be an increasing sequence of positive real numbers such that
$T_N\to \infty$ and $0 < C \leq \frac{T_N - N}{T_N} $ holds for some constant $C$ and all $N\geq 1$. 
We define $u_N : \mathbb{Z}\to \mathbb{R}$ by $u_N(j)= j \varphi_N(j)$, $j \in \mathbb{Z}$, where $\varphi_N:\mathbb{Z}\to \mathbb{R}$ is symmetric and such that $\varphi_N(j) = 1$ for $\vert j \vert \leq N$, 
$\varphi_N(j)= \frac{T_N-j}{T_N - N}$ if $N+1 \leq j \leq T_N$ and $\varphi_N(j)=0$ for $j > T_N$. For the sake of convenience we assume that $N$ is even in order to ensure $\frac{N}{2} \in \mathbb{N}$. In what follows we also drop the index $N$ in the notation of the
cut-off function, that is we just write $\varphi$ instead of $\varphi_N$.
Observe that the definition of $\varphi$ implies that $\vert \varphi(j) - \varphi(l)\vert \leq \frac{1}{T_N}|j-l|$, for $j,l \in \mathbb{Z}$. 

On one hand we have
\begin{equation}\label{eq:Gammaun}
\Gamma (u_N) (0) = 2 \sum\limits_{l=1}^{T_N} k(l) l^2 \varphi(l)^2 = 2 \left( \sum\limits_{l=1}^N k(l)l^2 + \sum\limits_{l=N+1}^{T_N} k(l)l^2 \varphi(l)^2 \right).
\end{equation}
On the other hand, due to $u_N$ being anti-symmetric, we have 
\begin{align*}
\Gamma_2(u_N)(0) \leq 2 \sum\limits_{j=1}^\infty \sum\limits_{l=1}^j k(j)k(l)\left[ \left( u_N(j+l)-u_N(j)-u_N(l)\right)^2 + \left(u_N(j-l) - u_N(j) + u_N(l) \right)^2\right].
\end{align*}
In order to estimate the $\Gamma_2$-term, we distinguish several cases (see Figure \ref{fig:3}). The strategy in each of the following cases will be the same. In fact, we will always estimate the respective 
part of the $\Gamma_2$-sum by some expression occurring in \eqref{eq:Gammaun} multiplied by some positive sequence converging to zero as $N\to \infty$. This implies that  for any given constant $\kappa>0$ there exists $N(\kappa) \in \mathbb{N}$ such that
\begin{equation*}
0<\Gamma_2 (u_N)(0) \le \kappa \Gamma (u_N)(0)
\end{equation*}
holds for any $N \geq N(\kappa)$, which will be sufficient to establish the claim. To simplify the following presentation, we 
use the notation $\delta(N)$ for a positive sequence, which may differ from line to line, such that $\delta(N)\to 0$ as $N\to \infty$.

Note that we can write for $j,l \in \mathbb{Z}$
\begin{equation}\label{eq:bracketforidentity}
\left( u_N(j+l) - u_N(j) - u_N(l) \right)^2 = \big( j \left( \varphi(j+l) - \varphi(j) \right) + l \left( \varphi(j+l) - \varphi(l) \right)\big)^2.
\end{equation}
\underline{I: $1\leq l \leq j$, $j+l \leq N$:} Here we have $\varphi(j+l)= \varphi(j)= \varphi(l) = 1$ and therefore all the 
corresponding summands are vanishing due to \eqref{eq:bracketforidentity}.\\
\underline{II: $1 \leq l \leq \frac{N}{2}$, $N +1 \leq j+l$, $\frac{N}{2}+1\leq j \leq T_N$:} Due to our choice of $\varphi$, we can control the bracket in \eqref{eq:bracketforidentity} by $\frac{j^2 l^2}{T_N^2}$ multiplied by some positive constant, and thus the part of $\Gamma_2(u_N)(0)$ that corresponds to the present case can be controlled by the sums
\begin{equation}\label{eq:calculationcaseII}
\sum\limits_{l=1}^N k(l) l^2 \frac{1}{T_N^2}\sum\limits_{j= \frac{N}{2}}^{T_N}  k(j) j^2 \leq \sum\limits_{l=1}^N k(l) l^2 \sum\limits_{j=\frac{N}{2}}^{T_N} k(j)\le \delta(N)\Gamma (u_N) (0) ,
\end{equation}
since $k \in l_1(\mathbb{Z})$.\\
\underline{III: $\frac{N}{2}+1 \leq l \leq j \leq T_N$:} We use the inequality $\left( a - b -c \right)^2 \leq 3 \left(a^2 + b^2 +c^2\right) $, $a,b,c \in \mathbb{R}$, and estimate the squares resulting from the left-hand side of \eqref{eq:bracketforidentity} and the previous inequality separately. We have
\begin{align*}
\sum\limits_{j=\frac{N}{2}+1}^{T_N} k(j) j^2 \varphi(j)^2 \sum\limits_{l=\frac{N}{2}+1}^j k(l) \leq \sum\limits_{j=\frac{N}{2}+1}^{T_N} k(j) j^2 \varphi(j)^2 \sum\limits_{l=\frac{N}{2}+1}^\infty k(l) \le \delta(N)\Gamma (u_N) (0).
\end{align*}
Similarly,
\begin{align*}
\sum\limits_{l=\frac{N}{2}+1}^{T_N} k(l) l^2 \varphi(l)^2 \sum\limits_{j=l}^{T_N} k(j)
 \leq \sum\limits_{l=\frac{N}{2}+1}^{T_N} k(l) l^2 \varphi(l)^2 \sum\limits_{j=\frac{N}{2}+1}^{T_N} k(j)
  \le \delta(N)\Gamma (u_N) (0).
\end{align*}
Since $(j+l)^2 \leq (2j)^2$ and $\varphi(j+l) \leq \varphi(j)$, by using the first estimate, we also obtain the desired estimate for the remaining third term. \\
\underline{IV: $1 \leq l \leq T_N +1$, $j\geq T_N$:} We have $\varphi(j)= \varphi(j+l)=0$ and thus it remains to estimate
\begin{align*}
\sum\limits_{l=1}^{T_N} k(l) l^2 \varphi(l)^2 \sum\limits_{j=T_N+1}^\infty k(j) \le \delta(N)\Gamma (u_N) (0). 
\end{align*}
\\
\underline{V: $T_N+1 \leq l \leq j$:} Now $\varphi(l)= \varphi(j)= \varphi(j+l) = 0$ and therefore the estimate is trivial.\\
\underline{A: $1\leq l \leq j \leq N$:} This case follows for the same reason as I.\\
\underline{B: $1 \leq l \leq N$, $N+1 \leq j \leq T_N$:} We now write
\begin{align*}
\left( u_N(j-l) - u_N(j) + u_N(l) \right)^2 = \big( j \left( \varphi(j-l) - \varphi(j)\right) + l \left( \varphi(l) - \varphi(j-l) \right) \big)^2
\end{align*}
and note that we can, as in the case II, control this expression by $\frac{j^2 l^2}{T_N^2}$ multiplied by some positive constant. This yields the same expression as in \eqref{eq:calculationcaseII} with the lower limit of the inner
sum replaced by $N+1$. Hence we get the desired estimate. \\
\underline{C: $N+1 \leq l \leq j , j \leq T_N$:} As in III we estimate the single terms, starting with
\begin{align*}
\sum\limits_{l=N+1}^{T_N} k(l) \sum\limits_{j=l}^{T_N}  & k(j) \left(j-l\right)^2 \varphi(j-l)^2  = \sum\limits_{l=N+1}^{T_N} k(l)\sum\limits_{j=1}^{T_N - l} k(j+l) j^2 \varphi(j)^2 \\
& \leq \sum\limits_{l=N+1}^{T_N} k(l) \sum\limits_{j=1}^{T_N} k(j) j^2 \varphi(j)^2 \le \delta(N)\Gamma (u_N) (0),
\end{align*}
where we applied that $k$ is monotone. Next, we estimate
\begin{align*}
\sum\limits_{j=N}^{T_N} k(j) j^2 \varphi(j)^2 \sum\limits_{l=N}^j k(l) \leq \sum\limits_{j=N}^{T_N} k(j) j^2 \varphi(j)^2 \sum\limits_{l=N}^{T_N} k(l)\le \delta(N)\Gamma (u_N) (0)
\end{align*}
and \begin{align*}
\sum\limits_{l=N}^{T_N} k(l) l^2 \varphi(l)^2 \sum\limits_{j=l}^{T_N}k(j) \leq \sum\limits_{l=N}^{T_N} k(l)l^2 \varphi(l)^2 \sum\limits_{j=N}^{T_N} k(j)\le \delta(N)\Gamma (u_N) (0).
\end{align*}
\\
\underline{D: $ T_N +1 \leq j$, $j-T_N \leq l \leq N$:} We have $\varphi(j)=0$ and hence there are two terms left to consider. First, obtain that
\begin{align*}
\sum\limits_{j=T_N+1}^{T_N+N} k(j) \sum\limits_{l=j-T_N}^{N} & k(l) \left(j-l\right)^2 \varphi(j-l)^2 = \sum\limits_{j=T_N + 1}^{T_N + N}k(j) \sum\limits_{l= - T_N}^{N-j} k(j+l) l^2 \varphi(l)^2\\
&= \sum\limits_{j=T_N+1}^{T_N + N} k(j) \sum\limits_{l=j-N}^{T_N} k(j-l) l^2 \varphi(l)^2 \\
&\leq \sum\limits_{j=T_N+1}^{T_N + N} k(j) \sum\limits_{l=j-N}^{T_N} k(l) l^2 \varphi(l)^2
\le \delta(N)\Gamma (u_N) (0) ,
\end{align*}
since in this case $j-l \ge l$ holds. In addition, since $\phi(l) = 1$ we have the upper bound
\begin{align*}
\sum\limits_{j=T_N + 1}^{T_N + N} k(j) \sum\limits_{l=j- T_N}^N k(l) l^2 \leq \sum\limits_{j=T_N + 1}^{T_N + N} k(j) \sum\limits_{l=1}^N k(l) l^2\le \delta(N)\Gamma (u_N) (0).
\end{align*}
\\
\underline{E: $N+1 \leq l \leq T_N$, $T_N + 1 \leq j \leq T_N + l$:} Again, we have $\varphi(j)=0$. On the one hand it holds
\begin{align*}
\sum\limits_{l=N+1}^{T_N} k(l) \sum\limits_{j=T_N + 1}^{T_N+l} & k(j) \left( j -l \right)^2 \varphi(j-l)^2 = \sum\limits_{l=N+1}^{T_N} k(l) \sum\limits_{j=T_N - l}^{T_N} k(j+l) j^2 \varphi(j)^2 \\
&\leq \sum\limits_{l=N+1}^{T_N} k(l) \sum\limits_{j=1}^{T_N} k(j) j^2 \varphi(j)^2\le \delta(N)\Gamma (u_N) (0)
\end{align*}
and on the other hand
\begin{align*}
\sum\limits_{l=N+1}^{T_N} k(l) l^2 \varphi(l)^2 \sum\limits_{j=T_N + 1}^{T_N + l} k(j) \leq \sum\limits_{l=N+1}^{T_N} k(l) l^2 \varphi(l)^2 \sum\limits_{j=T_N +1}^{2 \;T_N} k(j)\le \delta(N)\Gamma (u_N) (0).
\end{align*}
\\
\underline{F: $1\leq l \leq T_N$, $T_N + l +1 \leq j $:} Here we have $\varphi(j) = \varphi(j-l) = 0$. It remains to consider
\begin{align*}
\sum\limits_{l=1}^{T_N}k(l)l^2 \varphi(l)^2 \sum\limits_{j=T_N + l}^\infty k(j) \leq \sum\limits_{l=1}^{T_N} k(l) l^2 \varphi(l)^2 \sum\limits_{j=T_N}^\infty k(j)\le \delta(N)\Gamma (u_N) (0).
\end{align*}
\\
\underline{G: $T_N+1 \leq l$, $l \leq j \leq T_N + l$:} In this case $\varphi(j)$ and $\varphi(l)$ vanish. Hence, we only need to consider
\begin{align*}
\sum\limits_{l=T_N + 1}^\infty k(l) \sum\limits_{j=l}^{T_N + l}& k(j) \left( j - l \right)^2 \varphi(j-l)^2 = \sum\limits_{l=T_N + 1}^\infty k(l) \sum\limits_{j=1}^{T_N} k(j+l) j^2 \varphi(j)^2 \\
&\leq  \sum\limits_{l=T_N + 1}^\infty k(l) \sum\limits_{j=1}^{T_N} k(j) j^2 \varphi(j)^2\le \delta(N)\Gamma (u_N) (0).
\end{align*}
\\
\underline{H: $T_N+1 \leq l$, $j\leq l + T_N + 1$:} This case is obvious, since $\varphi(j-l) = \varphi(l) = \varphi(j) = 0$.
\end{proof}
\begin{bemerk1}
 {\em A careful inspection of the proof shows that the monotonicity assumption on $k$ can be weakened by assuming that there exists some $j_0 \in \N$, such that $k$ 
non-increasing on $\mathbb{N}\setminus\{1,...,j_0\}$. }
\end{bemerk1}

\newcommand{\etalchar}[1]{$^{#1}$}

%
%
%

\end{document}